\documentclass[12pt,twoside]{article} 
\usepackage[latin1]{inputenc}
\usepackage{graphicx}
\usepackage{enumitem}
\usepackage{latexsym}
\usepackage{amssymb}
\usepackage{epsfig}
\usepackage{xcolor}
\usepackage{amsmath}
\usepackage{float}
\allowdisplaybreaks 
\usepackage[round, sort]{natbib}

\numberwithin{equation}{section}

\newcommand{\Var}{\mbox{Var}}
\def\eps{\varepsilon}
\def \R{\mathbb{R}}
\def \N{\mathbb{N}}

\def\P{\mathbb{P}}

\def\er{\mathbb{R}}

\def\e{\varepsilon}

\def\beq{\begin{eqnarray*}}
\def\eeq{\end{eqnarray*}}
\def\farc{\frac}

\begin{document}

\title{\bf Analysis of gradual changes in nonparametric regression based on a new optimization method in the non-unique case
}

\author{ Marie Hu\v{s}kov\'{a}\\ \small Department of Statistics\\\small Charles University of Prague\!\!\!\! \and Natalie Neumeyer \\ \small Department of Mathematics \\\small University of Hamburg\and  Leonie Selk
\\ \small Department of Mathematics \\\small University of Hamburg}

\maketitle

\newtheorem{theo}{Theorem}[section]
\newtheorem{lemma}[theo]{Lemma}
\newtheorem{cor}[theo]{Corollary}
\newtheorem{rem}[theo]{Remark}
\newtheorem{prop}[theo]{Proposition}
\newtheorem{defin}[theo]{Definition}
\newtheorem{example}[theo]{Example}
\newtheorem{Assumption}{Assumption}

\begin{abstract}
Consider a nonparametric regression model with one-dimensional covariates and a continuous regression function. Assume that the regression function from the left of the covariate support starts equal to zero and then changes at some unknown point. Our aim is to estimate this gradual change point. We define and compare various consistent estimators based on a new general optimization method in the case where the aim is to estimate the largest minimization point of some objective function. We discuss rates of convergence and estimating the regression function based on the gradual change structure. Bootstrap bias approximation is discussed. Further applications in a two sample case are considered, where two continuous regression functions first equal and then change at some point of interest.  
\end{abstract}




\bigskip

\small

{\bf Acknowledgement.}  Many thanks  to Zden\v{e}k Hl\'{a}vka and Stanislav Hu\v{s}ek for discussion on the application. 
\normalsize

\section{Introduction}

We suggest a new general consistent estimation method for the largest minimizing point of  a population objective (criterion) function in the non-unique case. We then apply the new method to obtain several consistent estimators of a gradual change point in a nonparametric regression model. The motivation was the comparison of two nonparametric regression models with scalar covariates. For simplicity assume the interval $[0,1]$ for the covariates. Assume that from the observed samples it is suggested that the two continuous regression functions are the same on $[0,\vartheta_0]$, but differ starting at the unknown point $\vartheta_0$. The aim is to estimate the change point $\vartheta_0$. In applications, one often aims to compare two groups, such as the height of girls and boys in relation to age, as examined by \cite{Hlavka-Huskova}. If the covariate values are the same in both samples, it is easier to consider the differences of the observations. Then one obtains a nonparametric regression model, where the continuous regression function is zero on $[0,\vartheta_0]$ and at least on some interval $(\vartheta_0,\vartheta_0 +\Delta)$ non-zero. There are practical applications for this as well. For instance in medical science the change point can be the threshold dose of a medication, which is the smallest amount needed to cause a reaction in the patient's body. \\
How to estimate $\vartheta_0$ some ideas were to modify estimators based on  corresponding parametric models. There is a wide range of literature about detecting and estimating gradual changes in parametric models as we will discuss below. However, modification of those methods to the nonparametric case does not work straightforward.  Often one obtains a general optimization method (we consider the minimum case), where $\vartheta_0$ is the  {\sl maximal} value {\sl minimizing} the population objective function function $M(\vartheta)$ 
over $\vartheta$.  Let the empirical objective function $M_n(\vartheta)$ be a consistent estimator for $M(\vartheta)$ uniformly in $\vartheta$. Although one assumes uniform 
convergence, one cannot apply the argmin theorem in the non-unique minimum case. Using the {\sl maximal} value {\sl minimizing} the function 
$M_n(\vartheta)$ over $\vartheta$ generally does not give a consistent estimator. 
If the empirical objective function is defined as an arithmetic mean the estimator is called $M$-estimator and the classical argmax (argmin) theorems for consistency and asymptotic normality can be found for example in \cite{vanderVaart} or \cite{Kosorok}. But uniqueness of the argument of the maximization (minimization) is needed. In most of our applications the empirical objective function has a more complicated structure because it is based on some nonparametric estimator and the parameter $\vartheta$. In general this means that an infinite-dimensional nuisance parameter is estimated nonparametrically, and depending on the structure of empirical objective  function the procedure is called semiparametric $M$-estimation or in general semiparametric optimization. Asymptotic results under uniqueness assumption can be found in \cite{Kosorok}, \cite{Pakes-Pollard} or \cite{Chen-etal}, among others.  Asymptotic argmin results under non-uniqueness are given by
\cite{Ferger2021} and \cite{Ferger2024}, but are not applicable in our situation. 
As for our change point application we consider the case with parameter $\vartheta\in[0,1]$, where the function $\vartheta\mapsto M(\vartheta)$ is continuous, zero on $[0,\vartheta_0]$,  and strictly increasing on an interval $(\vartheta_0,\vartheta_0 +\Delta)$ with $\Delta>0$. We define a new estimator for $\vartheta_0$ and prove consistency. It is based on a formula related to monotone rearrangement of functions which are considered by \cite{Hardyetal} for instance. The new method is then applied to estimate the gradual (smooth) change point $\vartheta_0$ of a nonparametric regression model with several different objective functions $M$.
\\
Estimation of gradual change points in parametric location models were considered by \cite{Huskova-1998}, \cite{Horvath-Huskova}, and corresponding hypotheses tests by
\cite{Huskova-Steinebach-2000}, \cite{Huskova-Steinebach-2002} and \cite{Slaby}, among others. 
The parametric location model typically has the form $Y_i=\mu+\delta\cdot  ((i-m)/n)_+$ (with notation $z_+=\max(z,0)$) and can be interpreted as parametric regression model with fixed design points $x_i=i/n$ and change point $\vartheta_0$ with $m=\lfloor n\vartheta_0\rfloor$.  \cite{Jaruskova} generalized the form using polynomial regression function. 
Change point estimation for gradual changes in parametric regression models were considered by  \cite{Doering} in the fixed design case, and by
\cite{Doering-Jensen} for random covariates. In the random covariates case consistency of the change point estimator is also discussed in a misspecification context where the regression function does not belong to the parametric function class in \cite{Doering-SMSA}. The parametric regression formula is $(x-\vartheta_0)_+^q$ (with some generalizations). 
Detecting gradual changes in classical time series models were considered by \cite{Wang} and \cite{Vogt-Dette}, among others, and in functional time series models recently by \cite{Bastian-Dette} and \cite{Kirch-etal}.
\\
In Section \ref{sec2} we present the new general consistent estimation method. This is applied in Section \ref{sec3} for a gradual change point estimation in a nonparametric regression model. Here five different methods are discussed and consistency of the change point estimators is shown. In Section \ref{sec-m} nonparametric estimation of the regression function under shape constraints is discussed. A bootstrap bias reduction method and simulations are presented in Section \ref{sec-bootstrap}. Methods in the two-sample case are considered in Section \ref{sec-application} as well as a real data application. The proofs can be found in the appendix.

\section{General estimation method}\label{sec2}

For the general approach assume that $\vartheta_0$ is the maximal value for which a function $M(\vartheta)$ is minimized over $\vartheta\in [0,1]$ (i.e.\ $\vartheta_0=\sup\arg\min \{M(\vartheta)\mid\vartheta\in[0,1]\}$), and let $M_n(\vartheta)$ be a uniformly in $\vartheta$ consistent estimator for $M(\vartheta)$.
Although one assumes uniform convergence, one cannot apply the argmin theorem if the minimum of $M$ is not unique. And if one defines the estimator for $\vartheta_0$ as the maximal value for which $M_n(\vartheta)$ is minimized over $\vartheta\in [0,1]$, this will typically not lead to a consistent estimator (see Example \ref{example} below, the simulation Section \ref{sec-bootstrap}, or \cite{Ferger2024}). 
Let us assume that the minimum of $M(\vartheta)$ is equal to zero, and note that one can write 
\begin{eqnarray}\nonumber
\vartheta_0 &=& \inf\{\vartheta\in[0,1]\mid M(\vartheta)>0\}
\;=\; \sup\{\vartheta\in[0,1]\mid M(\vartheta)=0\}\\
&=& \int_0^1 I\{M(\vartheta)\leq 0\}\,d\vartheta,\label{M}
\end{eqnarray}
where $I\{\dots\}$ denotes the indicator function.
We define the estimator for the change point $\vartheta_0$ as
\begin{equation}\label{theta-hat}
\hat\vartheta_n= \int_0^1 I\{M_n(\vartheta)\leq t_n\}\,d\vartheta
\end{equation}
for a positive deterministic sequence $t_n\to0$. The following theorem gives consistency.

\begin{theo}\label{theo1}
Assume 
\begin{equation}\label{tn}
    \P\Big(\sup_{\vartheta\in [0,1]}|M_n(\vartheta)-M(\vartheta)|>t_n\Big)\longrightarrow 0\mbox{ for } n\to\infty,
\end{equation} 
and that  $M$ is continuous and non-decreasing on $[0,1]$,  zero on $[0,\vartheta_0]$, and strictly increasing on  $(\vartheta_0,\vartheta_0+\Delta)$ for some $\Delta>0$. Then  
$$\hat\vartheta_n=\vartheta_0+O_{\P}(\delta_n)$$ for $\delta_n=M^{-1}(2t_n)-\vartheta_0$. 
\end{theo}

For the notation note that $M$ is invertible on $[\vartheta_0,\vartheta_0+\Delta]$ with a continuous inverse function $M^{-1}$ on $[0, M(\vartheta_0+\Delta)]$, and for $n$ large enough $2t_n$ will be an element of $[0, M(\vartheta_0+\Delta)]$. We obtain $\delta_n=o(1)$ and thus consistency of the estimator.  
The proof of the theorem is given in the appendix. This result is very general. We will apply it to estimate gradual changes in a nonparametric regression function with independent data. But it could also be applied for time series data or in very different models and estimation problems. 

\begin{example}\label{example}\rm We consider a simple non-random example to show that the argmin procedure does not work, but the new suggested method. Let 
\begin{eqnarray*}
    M(\vartheta)&=&\left\{ \begin{array}{l} 0,\quad \vartheta\in [0,\frac12]\\ \vartheta-\frac12,\quad \vartheta\in (\frac12, 1]\end{array}\right. \\
    M_n(\vartheta)&=&
    \left\{ \begin{array}{l} \frac{\vartheta}{1+\frac{n}{2}},\quad \vartheta\in [0,\frac12+\frac1n]\\ \vartheta-\frac12,\quad \vartheta\in (\frac12+\frac1n, 1].\end{array}\right.
    \end{eqnarray*}
 Then $\sup_\vartheta|M_n(\vartheta)-M(\vartheta)|$   converges to zero with rate $\frac1n$, but $\vartheta_0=\max\arg\min_\vartheta M(\vartheta)=\frac12$, and $\max\arg\min_\vartheta M_n(\vartheta)=\arg\min_\vartheta M_n(\vartheta)=0$ for all $n$. If one chooses $t_n>\frac1n$, then 
 $$\int_0^1 I\{M_n(\vartheta)\leq t_n\}\,d\vartheta=\int_0^{\frac12+\frac1n}I\{\vartheta\leq t_n(1+\textstyle{\frac{n}{2}})\}\, d\vartheta+t_n-\frac1n=\frac12+t_n$$
 converges to $\vartheta_0=\frac12$ for $t_n\to 0$, $n\to\infty$. 
\end{example}

\begin{rem}\rm The definition of the estimator $\hat\vartheta_n$ is related to increasing rearrangements of functions. Consider the function 
$\Phi(f)(z)= \int_a^b I\{f(x)\leq z\}\, dx+a$ for $z\in\mathbb{R}$
defined for any measurable function $f:[a,b]\to\mathbb{R}$.  
For a strictly increasing function $f$, the function $\Phi(f)|_{[f(a),f(b)]}$ is just the inverse $f^{-1}$.  
The increasing rearrangement of a non-monotone function $h:[0,1]\to\er$ is defined as $\Phi(\Phi(h)|_{[h(0),h(1)]})|_{[0,1]}$. Large theory overviews for monotone rearrangements of functions can be found in \cite{Hardyetal} and \cite{Rakotoson2025}. The procedure has been applied in statistics  to estimate monotone functions. 
Our estimator has the form $\hat\vartheta_n=\Phi(M_n)(t_n)$.
In particular for the method (3) in Subsection \ref{increasing regression} suggested below, those rearrangements for estimating monotone nonparametric regression functions have been considered by \cite{Dette-etal2006}, \cite{Neumeyer2007} and \cite{AnevskiFougeres2019}.
 But those results are not  applicable in our case where we apply the generalized inverse of the rearrangement, and the regression function is zero on the interval $[0,\vartheta_0]$. 
\end{rem}

\begin{rem}\rm
The consistency proof of Theorem \ref{theo1} also works replacing the assumption ``$M$ is zero on $[0,\vartheta_0]$'' by ``$M\leq 0$  on $[0,\vartheta_0]$''. In the case $M(\vartheta)< 0$ for $\vartheta\in [0,\vartheta_0)$ this means one estimates the root of a non-decreasing continuous function $M$.
\end{rem}

\section{Gradual change point estimation}\label{sec3}

We consider a nonparametric regression model
\begin{equation}\label{model}
Y_i=m(X_i)+\eps_i,\quad i=1,\dots,n,
\end{equation}
with independent identically distributed random variables $(X_i,\varepsilon_i)$, $i=1,\dots,n$, where the support of the distribution of $X_i$ is $[0,1]$, and we assume $E[\eps_i\mid X_i]=0$, $\Var(\eps_i)=\sigma^2\in (0,\infty)$. We use the notation $\P^X$ for the distribution of $X_i$ and $F_X$ for the corresponding cdf, and assume a positive density $f_X$ on $[0,1]$.
We assume that the regression function $m$ is continuous with $m(x)\equiv 0$ on $[0,\vartheta_0]$, and $m\neq 0$ on some interval $(\vartheta_0,\vartheta_0 +\Delta)$ with  $\Delta \in (0,1-\vartheta_0]$. 
The gradual change point $\vartheta_0$ is unknown and of interest. For some of the different methods considered below more assumptions are needed.   

\begin{rem}\rm
A typical example of what $m$ could look like is given by $m(x)=(x-\vartheta_0)_+^q$ for some $q>0$. Those are e.\,g.\ considered in \cite{Doering-Jensen} and in \cite{Doering}, where also estimators for the degree of smoothness $q$ are proposed. In the paper at hand we do not assume this parametric form in general. However, in Subsections \ref{method4}, \ref{method5} and Section \ref{chap:regression} we propose change point estimation methods and a regression estimator that are based on $m(x)=g(x)(x-\vartheta_0)^q_+$ for some suitable function $g$. There, and also where we use $m(x)=(x-\vartheta_0)^q_+$ as an example, we assume that $q$ is known and refer to \cite{Doering-Jensen} for procedures to estimate $q$. In Subsection \ref{examplerates} we give the rates of our estimators with $m(x)=(x-\vartheta_0)^q_+$ for different values of $q$. It can be seen that the rate of convergence of the change point estimator improves with decreasing $q$, since smaller values of $q$ result in more distinct changes. The extreme case $q
=0$ corresponds to a jump in the regression function. The influence of the size of $q$ can also be seen in our simulation results, see Figure \ref{fig:compRegr}.
\end{rem}

\subsection{Several methods to estimate the change point}\label{sec:methods}

We now consider various possibilities to define $M$ and $M_n$ under different assumptions on the regression function $m$ on $(\vartheta_0,1]$. 

\subsubsection{Method (1): positive regression function}\label{sec-Mn-1}\label{method1}

Assume that the continuous $m$ is positive on some interval $(\vartheta_0,\vartheta_0 +\Delta)$ and non-negative on $[\vartheta_0 +\Delta,1]$. 
Let $$M(\vartheta)=\int_{\vartheta_0}^\vartheta m(x)\,dF_X(x)\, I\{\vartheta>\vartheta_0\},$$ which is the expectation of 
\begin{equation}\label{M_n-1}
M_n(\vartheta)=\frac{1}{n}\sum_{i=1}^n Y_iI\{X_i\leq\vartheta\},
\end{equation}
and fulfills the assumptions of Theorem \ref{theo1}. 
Combining Example 2.6.1 and  exercise 9 (p.\ 151) by \cite{vanderVaartWellner} it follows that the function class
$\{(x,y)\mapsto I\{x\leq\vartheta\}  \mid \vartheta\in[0,1]\}$ is a VC-subgraph-class. By Lemma 2.6.18(vi) in \cite{vanderVaartWellner} it follows that the class $\{(x,y)\mapsto y I\{x\leq\vartheta\}  \mid \vartheta\in[0,1]\}$ is a VC-subgraph-class. Thus this class is $P$-Donsker, where $P$ denotes the distribution of $(X_1,Y_1)$, by Theorem 2.8.3 of \cite{vanderVaartWellner}), because $E[Y_i^2]<\infty$. We derive that $\{\sqrt{n}(M_n(\vartheta)-M(\vartheta))\mid \vartheta\in[0,1]\}$ converges weakly to a centered Gaussian process, and one obtains validity of (\ref{tn}) with $t_n=n^{-1/2}/c_n$ for every positive sequence $c_n$ converging to zero. 
Then the estimator $\hat\vartheta_n$ defined in (\ref{theta-hat}) is consistent. 

\begin{example}\rm
    Assume $X_1\sim \mathcal{U}[0,1]$.

(1) Let $m(x)=(x-\vartheta_0)_+^q$ for $q>0$, then we obtain $M(\vartheta)= (\vartheta-\vartheta_0)_+^{q+1}/(q+1)$ and $M^{-1}(t)=\vartheta_0+((q+1)t)^{1/(q+1)}$  and the rate is $\delta_n\sim t_n^{1/(q+1)}$.

(2) Let $m(x)=2(x-\vartheta_0)e^{(x-\vartheta_0)^2}I\{x>\vartheta_0\}$, then $M(\vartheta)=(e^{(\vartheta-\vartheta_0)^2}-1)I\{\vartheta>\vartheta_0\}$ and $M^{-1}(t)=\vartheta_0+(\log(t+1))^{1/2}$ and the rate is $\delta_n=(\log(t_n+1))^{1/2}$.
\end{example}

\begin{rem}\rm
    Compared to the methods (2)--(5) discussed in the next subsections method (1) has the weakest assumptions. The objective function $M$ needs to be strictly increasing on $(\vartheta_0,\vartheta_0+\Delta)$. This holds if $m$ and the assumed density $f_X$ are positive on this interval. It is even not necessary to assume an absolute continuous distribution $\P^X$. No smoothness assumptions on $m$ are needed. 
\end{rem}


In Figures \ref{fig1} and \ref{fig2} we demonstrate the advantage of the new method compared to the non-consistent argmin procedure.
In Figure \ref{fig1} the graphic at the top left shows the true regression function $m(x)=(x-\vartheta_0)_+$ and the scatter plot of one  data set $(X_i,Y_i)$, $i=1,\dots,n$, for $n=200$ and true change point $\vartheta_0=0.4$. The covariates are uniformly distributed and the errors normally distributed with standard deviation $\sigma=0.2$. The graphic at the top right shows the true function $M(\vartheta)=(\vartheta-\vartheta_0)^2I\{\vartheta>\vartheta_0\}$. Based on the sample the graphic on the bottom left shows the function $M_n$. The argmin procedure (largest value minimizing $M_n$) does not give a suitable estimator for  $\vartheta_0=0.4$. The graphic on the bottom right shows the increasing rearrangement of $M_n$ from (\ref{M_n-1}). The horizontal line is the value $t_n=\sigma n^{-0.49}$, and the estimator $\hat\vartheta_n$ defined in (\ref{theta-hat})  is the value, where the line intersects with the increasing rearrangement of $M_n$. This seems more suitable than applying the argmin procedure. 

 \begin{figure}[h]
     \centering
     \includegraphics[width=10cm]{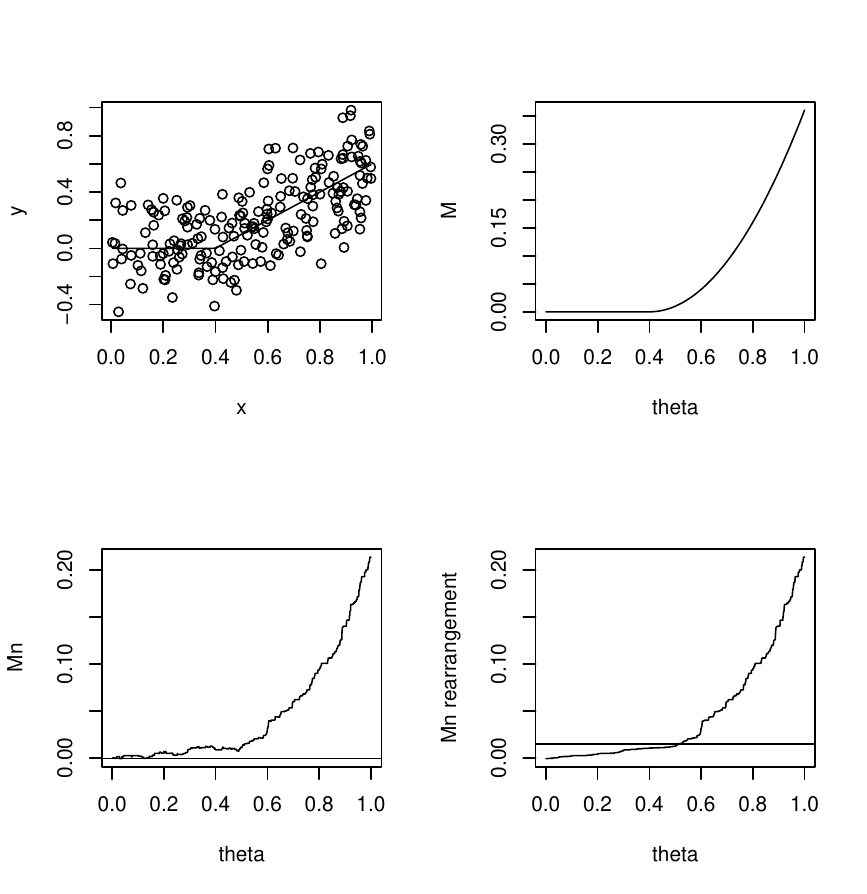}
     \caption{
     \it Top left: scatter plot of observations and true regression function; top right: population objective function $M$; bottom left: empirical objective function $M_n$; bottom right: increasing rearrangement of $M_n$ and horizontal line $t_n$.}
     \label{fig1}
 \end{figure}

While the graphics in Figure \ref{fig1} are based on one sample, now we present 1000 replications to demonstrate that the  argmin procedure does not lead to a consistent estimator. 
In Figure \ref{fig2} we use the same model with $\vartheta_0=0.4$ and choice of $t_n$. The graphic at the top left shows again the true regression function and the scatter plot of one  data set $(X_i,Y_i)$, $i=1,\dots,n$, for $n=200$. 
 The graphic at the bottom left shows  the estimated density of the estimators $\hat\vartheta_n$ 
 based on the 1000 replications. For the same data the graphic at the bottom right shows the estimated density of estimators $\bar\vartheta_n=\sup\{\vartheta\in [0,1]: M_n(\vartheta)=\inf_{\eta\in[0,1]}M_n(\eta)\}$ giving the largest minimizing value of $M_n$.  The graphic demonstrates that this procedure is not consistent as the density has two modes and more often estimates a too small value. The graphic on top right shows the boxplots of the 1000 estimators (left boxplot for $\hat\vartheta_n$, right boxplot for $\bar\vartheta_n$). Although we have proved consistency of $\hat\vartheta_n$ a bias approximation would be advantageous as the estimators are too large which is often the case in change point estimation. We discuss bootstrap bias approximation in Section \ref{sec-bootstrap}.

 \begin{figure}[h]
     \centering
     \includegraphics[width=10cm]{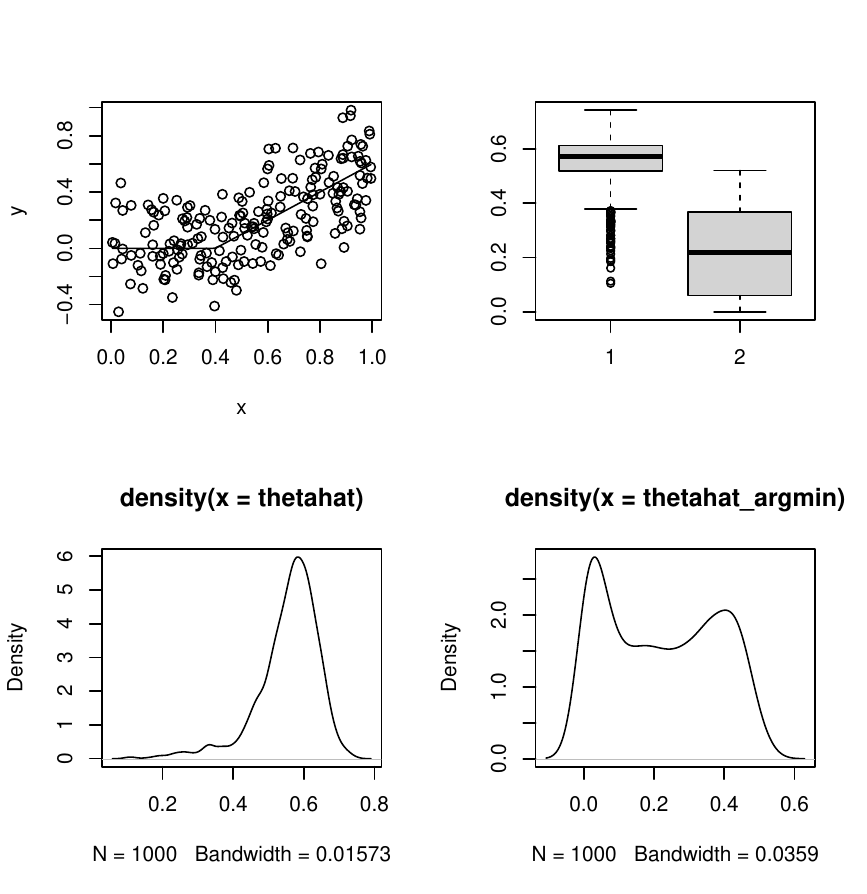}
       \caption{
       \it Top left: scatter plot of observations and true regression function; top right: boxplots (left $\hat\vartheta_n$, right argmin $\bar\vartheta_n$);  bottom left: estimated density of $\hat\vartheta_n$; bottom right: estimated density of the argmin estimator $\bar\vartheta_n$.}
     \label{fig2}
 \end{figure}


\subsubsection{Method (2): non-zero regression function}\label{non-zero}\label{method2}

Assume that the $\P^X$-measure of $\{x\in(\vartheta_0,\vartheta_0+\Delta): m(x)=0\}$ is zero. Let $\hat m_n$ be a uniformly consistent nonparametric estimator for the regression function $m$. 
Let $$M(\vartheta)=\int_{\vartheta_0}^\vartheta m^2(x)\,dF_X(x)\, I\{\vartheta>\vartheta_0\},$$ which is continuous on $[0,1]$ and zero on $[0,\vartheta_0]$. We estimate it by 
$$M_n(\vartheta)=\frac{1}{n}\sum_{i=1}^n (\hat m_n(X_i))^2I\{X_i\leq\vartheta\}. $$
To show consistency of $\hat\vartheta_n$ assume
 \begin{equation}\label{assumption2}
   P((\hat m_n-m)\in\mathcal{H})\longrightarrow 1 \mbox{ for } n\to\infty \mbox{ with a uniformly bounded Donsker class }\mathcal{H},
\end{equation}
which is often needed for empirical process theory based on residuals. For instance \cite{AkritasVanKeilegom} showed in their proof of Lemma 1 that
$P((\hat m_n-m)\in C_1^{1+\delta}([0,1]))\to 1$ for a Nadaraya-Watson estimator $\hat m$, where $C_1^{1+\delta}([0,1])$ with $\delta\in (0,1]$ is the class of all differentiable functions $d$ defined on $[0,1]$ such that $\max(\sup_x|d(x)|,\sup_x|d'(x)|)+\sup_{x,x'}\frac{|d'(x)-d'(x')|}{|x-x'|^\delta}$ $\leq 1$, which is a uniformly bounded Donsker class by Corollary 2.7.2 in \cite{vanderVaartWellner}. The result of \cite{AkritasVanKeilegom} is based on the smoothness assumption that $m$ is twice continuously differentiable and that $F_X$ is twice continuously differentiable and $\inf_{x}f_X(x)>0$. Further assume
\begin{equation}\label{assumption1}
   \sup_{\vartheta\in[0,1]}\left| \int_0^\vartheta (\hat m_n^2(x)-m^2(x))\, dF_X(x)\right|=O_{\P}(n^{-1/2}).
\end{equation}
This is also typically the case because one can rewrite the term as
\begin{eqnarray*}
\int_0^\vartheta (\hat m_n(x)-m(x))^2\, dF_X(x) +
    2\int_0^\vartheta (\hat m_n(x)-m(x))m(x)\, dF_X(x).
\end{eqnarray*}
The first term can be upper bounded by $E[(\hat m_n(X)-m(X))^2]$, which typically has the desired rate, again under smoothness assumptions.  The second term for kernel estimators under suitable assumptions has the desired rate because the integral smooths the estimator. 

The next lemma shows validity of assumption (\ref{tn}) and thus Theorem \ref{theo1} gives consistency of $\hat\vartheta_n$ with choosing $t_n=n^{-1/2}/c_n$ for every positive sequence $c_n$ converging to zero. 

\begin{lemma}\label{lem-non-zero-regression}
Under assumptions (\ref{assumption1}) and (\ref{assumption2}) it holds that
$$\sup_{\vartheta\in [0,1]}|M_n(\vartheta)-M(\vartheta)|=O_{\P}(n^{-1/2}).$$
\end{lemma}

The proof is given in the appendix.

\begin{example}\label{example-non-zero}\rm
    Assume $X\sim \mathcal{U}[0,1]$.

(1) Let $m(x)=(x-\vartheta_0)_+^q$ for $q>0$ , then we obtain $M(\vartheta)= (\vartheta-\vartheta_0)_+^{2q+1}/(2q+1)$ and $M^{-1}(t)=\vartheta_0+((2q+1)t)^{1/(2q+1)}$ and the rate is $\delta_n\sim t_n^{1/(2q+1)}$.

(2) Let $m(x)=\sqrt{2(x-\vartheta_0)}e^{\frac{1}{2}(x-\vartheta_0)^2}I\{x>\vartheta_0\}$, then $M(\vartheta)=(e^{(\vartheta-\vartheta_0)^2}-1)I\{\vartheta>\vartheta_0\}$ and $M^{-1}(t)=\vartheta_0+(\log(t+1))^{1/2}$ and the rate is $\delta_n=(\log(t_n+1))^{1/2}$.

\end{example}

\subsubsection{Method (3): increasing regression function}\label{increasing regression}\label{method3}

Assume that $m$ is strictly increasing on $(\vartheta_0,\vartheta_0 +\Delta)$ and non-decreasing on $[\vartheta_0 +\Delta,1]$. Then one can use $$M(\vartheta)=m(\vartheta) \mbox{ and }M_n(\vartheta)=\hat m_n(\vartheta)$$ for some uniformly consistent nonparametric regression estimator $\hat m_n$. For (\ref{tn}) assume that
$$\sup_{\vartheta\in [0,1]}|\hat m_n(\vartheta)-m(\vartheta)|=O_{\P}(\beta_n).$$
Then one can choose $t_n=\beta_n/c_n$ for any positive sequence $c_n\to 0$ and Theorem \ref{theo1} holds. 

The optimal rate  is $\beta_n=((\log n)/n)^{\alpha/(2\alpha+1)}$ for a nonparametric $(k,\delta)$-H\"older regression function $m$ with $\alpha=k+\delta$, $k\in\N_0$, $\delta\in (0,1]$, see \cite{Stone}. 
For instance this holds for local polynomial regression estimators under suitable assumptions; in particular for twice continuously differentiable regression functions one obtains $\beta_n=b_n^2+( (\log n)/(nb_n))^{1/2}$ for a bandwidth sequence $b_n\to 0$.  
Uniform rates can be found in \cite{AkritasVanKeilegom} for a Nadaraya-Watson type estimator under the assumption that $m$ and $F_X$ are twice continuously differentiable and $\inf_{x}f_X(x)>0$, or in \cite{Antoch-etal} for local linear estimator.  \cite{Hansen} shows the uniform rate for local constant and local linear estimators in the more general case of strong mixing observations under the same smoothness assumptions. Because in this subsection we assume monotonicity one can also use increasing regression estimators, see \cite{Groeneboom-Jongbloed}.

\begin{example}\rm 
(1) Let $m(x)=(x-\vartheta_0)_+^q$ for $c>0$ , then we obtain $m^{-1}(t)=\vartheta_0+t^{1/q}$ and the rate is $\delta_n=t_n^{1/q}$.
In the case $q\in (0,1)$ this can lead to a faster rate of convergence than $\beta_n$, whereas for $q>1$ the rate will be slower. 

(2) Let $m(x)=(\exp(x-\vartheta_0)-1)I\{x\in(\vartheta_0,1]\}$, then we have 
$m^{-1}(t)=\vartheta_0+\log (t+1)$  and obtain the slow rate $\delta_n=\log(t_n+1)$.
\end{example}

\subsubsection{Method (4): positive right $q$-th derivative}\label{sec-deriv}\label{method4}

In this subsection we assume that 
$m(x)=(x-\vartheta_0)^q_+\, g(x)$ for some Lipschitz function $g:[0,1]\to\mathbb{R}_0^+$ with $q>0$ known and $g(\vartheta_0)=m^{(q)}(\vartheta_0+)> 0$. 
Then we define the estimator
 \begin{equation}\label{hat-g}
     \hat g_{\vartheta}(x)=\arg\min_{c\in\R}\sum_{i=1}^n \left(Y_i-c(X_i-\vartheta)_+^q\right)^2\kappa\Big(\frac{X_i-x}{d_n}\Big)=
\frac{\sum_{i=1}^n Y_i\kappa(\frac{X_i-x}{d_n})(X_i-\vartheta)_+^q}{\sum_{i=1}^n \kappa(\frac{X_i-x}{d_n})(X_i-\vartheta)_+^{2q}},
 \end{equation}
where $\kappa$ is a one-sided kernel with bounded support $[0,C_\kappa]$, that is Lipschitz continuous, and $d_n$ is a sequence of bandwidths with $d_n\to 0$ and $\log(n)/(nd_n)\to 0$. Then for $x>\vartheta$, $\hat g_{\vartheta}(x)$ estimates $ g(x)\frac{(x-\vartheta_0)_+^q}{(x-\vartheta)_+^q}$ if $\vartheta\geq\vartheta_0$ or $\vartheta<\vartheta_0\leq x$, and $0$ if $\vartheta<x<\vartheta_0$. 
Thus 
\begin{equation}
    \label{hat-m-theta}
    \hat m_\vartheta(x)= \hat g_{\vartheta}(x)(x-\vartheta)^q_+
\end{equation} estimates
$m(x)I\{x>\vartheta\}$, and
$$ M_n(\vartheta)=\frac 1n\sum_{i=1}^n \left(Y_i-\hat m_\vartheta(X_i)\right)$$ estimates
\begin{eqnarray*}
M(\vartheta)&=&\int(m(x)-m(x)I\{x>\vartheta\})\,dF_X(x)=\int_{\vartheta_0}^\vartheta m(x)\,dF_X(x)I\{\vartheta>\vartheta_0\},
\end{eqnarray*}
the same objective function as in Subsection \ref{method1}. We need the following assumptions. 
\begin{align}
\label{ass1} &\!\text{
    Let the distribution  of the covariates $X$ have a continuous density $f_X$}\\ &\text{that is bounded from below on $[0,1]$. }\nonumber\\\nonumber \\
\label{ass2}&\!\text{
   Let the errors $\varepsilon$ fulfill  $\forall l\geq 2\ \exists C_l>0:E[|\varepsilon|^l]\leq C_la^{2(l-1)}$}\\&\text{for some constant $a$. }\nonumber
\end{align}

\begin{lemma}\label{lem-deriv-pos} Let $m$ fulfill the assumptions described at the beginning of Subsection \ref{method4}, and assumptions \eqref{ass1} and \eqref{ass2} hold. 
Then, with $d_n\sim n^{-(2+2(q\wedge 1))/(2+6(q\wedge 1))}$ it holds
\begin{align*}
\sup_{\vartheta\in[0,1]}|M_n(\vartheta)-M(\vartheta)|=&\ O_{\P}\Big(\sqrt{\frac{\log(n)}{nd_n}}\Big)+O(d_n^{(q\wedge 1)/(1+(q\wedge 1))}\log(n)^{(q\wedge 1)/(1+(q\wedge 1))})+O_{\P}(n^{-\frac 12})\\
=&\ O_{\P}(n^{-2(q\wedge 1)/(2+6(q\wedge 1))}\log(n)^{\frac 12}).
\end{align*}
\end{lemma}

The proof is given in the appendix.
\\
For the estimator $ \hat\vartheta_n$ one can choose $t_n=(\log n)^{1/2}/(n^{2(q\wedge 1)/(2+6(q\wedge 1))}c_n)$
for any positive $c_n\to 0$, and the result from Theorem \ref{theo1} holds.

\subsubsection{Method (5): non-zero right $q$-th derivative}\label{method5}

In this subsection we assume that 
$m(x)=(x-\vartheta_0)^q_+\, g(x)$ for some Lipschitz function $g:[0,1]\to\mathbb{R}_0^+$ with $q>0$ known and $g(\vartheta_0)=m^{(q)}(\vartheta_0+)\neq 0$. 
Then one can define
$$M_n(\vartheta)=\frac 1n\sum_{i=1}^n \left(Y_i-\hat m_{\vartheta}(X_i)\right)^2-\hat\sigma_n^2$$
with $\hat m_\vartheta$ from (\ref{hat-m-theta}). Further let $\hat\sigma_n^2$ be a consistent estimator for the error variance $\sigma^2=E[\varepsilon_1^2]$. We assume that $|\hat\sigma_n^2-\sigma^2|=O_{\P}(n^{-\frac 12})$, which is typically fulfilled, e.\,g.\ with $\hat\sigma_n^2:=\frac 1n\sum_{i=1}^n (Y_i-\hat m_n(X_i))^2$ for some consistent nonparametric regression estimator $\hat m_n$.
Then $M_n(\vartheta)$ estimates 
$$M(\vartheta)=\int_{\vartheta_0}^{\vartheta}m^2(x)dF_X(x)I\{\vartheta>\vartheta_0\},$$ 
which is the same objective function as in Subsection \ref{non-zero}.

\begin{lemma}\label{lem-deriv}
 Let $m$ fulfill the assumptions described at the beginning of Subsection \ref{method5}, and assumptions \eqref{ass1} and \eqref{ass2} from Subsection \ref{method4} hold.  Then, under the assumption $|\hat\sigma_n^2-\sigma^2|=O_{\P}(n^{-\frac 12})$ and with $d_n\sim n^{-(1+2(q\wedge 1))/(1+4(q\wedge 1))}$ it holds
\begin{align*}
&\sup_{\vartheta\in[0,1]}|M_n(\vartheta)-M(\vartheta)|\\
=&\ O_{\P}\Big(\frac{\log(n)}{nd_n}\Big)+O(d_n^{2(q\wedge 1)/(1+2(q\wedge 1))}\log(n)^{2(q\wedge 1)/(1+2(q\wedge 1))})+O_{\P}(n^{-\frac 12})+O(|\hat\sigma_n^2-\sigma^2|)\\
=&\ O_{\P}(n^{-2(q\wedge 1)/(1+4(q\wedge 1))}\log(n)).
\end{align*}
\end{lemma}

The proof is given in the appendix. Consistency of $\hat\vartheta_n$ from Theorem \ref{theo1} holds with $t_n=(\log n)/(n^{2(q\wedge 1)/(1+4(q\wedge 1))}c_n)$ for positive $c_n\to 0$.

\subsubsection{Examples rates of convergence}\label{examplerates}

Consider the uniform covariate case $X\sim \mathcal{U}[0,1]$ and the regression function   $m(x)=(x-\vartheta_0)_+^q$ for $q>0$. We obtain the following rates.

\begin{itemize}
   
 \item  Method (1): $t_n\gg n^{-1/2}$, $\delta_n=t_n^{1/(q+1)}$

 \item  Method (2): $t_n\gg n^{-1/2}$, $\delta_n=t_n^{1/(2q+1)}$

 \item Method (3):  $m$ is in the $(0,q)$-H\"older-space for $q\in (0,1]$, and in the $(k,\delta)$-H\"older-space with $k=\underline{q}$ the greatest integer smaller than $q$, and $\delta=q-\underline{q}$, if $q>1$. Then  using the  optimal rate one has $t_n\gg ((\log n)/n)^{q/(2q+1)}$ and obtains $\delta_n=t_n^{1/q}$.

 \item Method (4): choose $d_n\sim n^{-\farc{2+2(q\wedge 1)}{2+6(q\wedge 1)}}$,  then $t_n\gg n^{-\frac{2(q\wedge 1)}{2+6(q\wedge 1)}}$, $\delta_n=t_n^{1/(q+1)}$

 \item Method (5): choose $d_n\sim n^{-\farc{1+2(q\wedge 1)}{1+4(q\wedge 1)}}$,  then $t_n\gg n^{-\frac{2(q\wedge 1)}{1+4(q\wedge 1)}}$, $\delta_n=t_n^{1/(2q+1)}$

\end{itemize}

To compare the rates of convergence in Table \ref{example-table} below we ignore $c_n$ and factors $\log n$. Note that although the rate of $M_n$ is slower in method (3), the rate of $\hat\vartheta_n$ is faster compared to the other methods. The derived rates of methods (4) and (5) are rather slow based on the rates of the empirical objective function in Lemmas \ref{lem-deriv-pos} and \ref{lem-deriv}.

\begin{table}[h]
   \begin{center}
   \begin{tabular}{l|c|c|c|c|c|}
method & (1) & (2)& (3)& (4)& (5)\\\hline\hline
$q=0.5$ &$n^{-1/3}$ & $n^{-1/4}$ & $n^{-1/2}$ & $n^{-2/15}$& $n^{-1/6}$\\
$q=1$& $n^{-1/4}$ & $n^{-1/6}$ & $n^{-1/3}$ & $n^{-1/8}$ & $n^{-2/15}$ \\
$q=2$&  $n^{-1/6}$ & $n^{-1/10}$ & $n^{-1/5}$& $n^{-1/12}$ & $n^{-2/25}$
\end{tabular}
\end{center}
\caption{\it Examples approximately rates of convergence of methods (1)--(5) for $m(x)=(x-\vartheta_0)_+^q$.\label{example-table}}
\end{table}

\subsection{Fixed design case}\label{fixed-design}

For application examples it might be useful also to consider the case of deterministic design, for which the theory has to be adjusted. Consider the model 
\begin{equation*}
Y_{i,n}=m(x_{i,n})+\eps_{i,n},\quad i=1,\dots,n,
\end{equation*}
with independent identically distributed random variables $\varepsilon_{i,n}$, $i=1,\dots,n$, $n\in\mathbb{N}$, with $E[\eps_{i,n}]=0$ and existing variance.  All considered methods (1)--(5) can be adjusted for this deterministic design case under suitable assumptions on the fixed design. 

Let us first consider method (1) from Subsection \ref{method1}, and assume the structure $F_X(x_{i,n})=\frac{i}{n}$ for some cdf $F_X$  with a positive density on $[0,1]$. This assumption is based on \cite{Sacks-Ylvisaker}. Then 
\begin{eqnarray*}
M_n(\vartheta)&=&\frac{1}{n}\sum_{i=1}^n Y_{i,n}I\{x_{i,n}\leq\vartheta\}\\
&=& \frac{1}{n}\sum_{i=1}^n m(F_X^{-1}(\frac{i}{n}))I\{\frac{i}{n}\leq F_X(\vartheta)\} +\frac{1}{n}\sum_{i=1}^n \eps_{i,n}I\{x_{i,n}\leq\vartheta\}.
\end{eqnarray*}
The first term is a Riemann-sum and converges with rate $n^{-1}$ to $$M(\vartheta)=\int_0^{F_X(\vartheta)} m(F_X^{-1}(x))\, dx=\int_0^\vartheta m(x)\,dF_X(x)=\int_{\vartheta_0}^\vartheta m(x)\,dF_X(x)I\{\vartheta>\vartheta_0\}.$$ For the second term let $\eps_1,\eps_2,\dots$ be iid with the same distribution as the triangular arrays $\eps_{i,n}$, $i=1,\dots,n$, $n\in\mathbb{N}$. Then the second term has the same distribution as 
$$\frac{1}{n}\sum_{i=1}^{\lfloor nF_X(\vartheta)\rfloor} \eps_{i}=O_{\P}(n^{-1/2})$$
uniformly in $\vartheta\in[0,1]$ by the partial sum Donsker result. 

For methods (2) and (3) from Subsections \ref{method2} and \ref{method3} one uses uniformly consistency results for the nonparametric regression estimator. For instance uniform rates of convergence of kernel and local polynomial estimators in the fixed design case are shown by \cite{Ioannides} and  \cite{Tsybakov}. Further for arithmetic means only depending on the covariates as in the proof of Lemma \ref{lem-non-zero-regression} one can apply uniform Riemann-sum approximations instead of empirical process theory. Similar results for  methods (4) and (5) can be derived in the fixed design case.

\section{Estimation of the regression function}\label{chap:regression}\label{sec-m}

In this section we assume that
$m(x)=(x-\vartheta_0)_+^q\, g(x)$ for some Lipschitz function $g:[0,1]\to\mathbb{R}_0^+$ with $g(\vartheta_0)\neq 0$ as in Subsections \ref{method4} and \ref{method5}.  Under the assumption of this regression function structure one would like to apply estimators which have the same structure. This is always of interest for shape constraints on regression functions as considered by \cite{Groeneboom-Jongbloed}, among others. A classical  nonparametric regression estimator will not have the desired shape. So based on $\hat g_\vartheta$ from (\ref{hat-g}) we define the estimator
$$\hat m_n(x)=(x-\hat\vartheta_n)_+^q\,\hat g_{\hat\vartheta_n}(x)$$
which is continuous on $[0,1]$, and zero on $[0,\hat\vartheta_n]$. 
The next theorem gives uniform consistency of the structured regression estimator using one of the consistent change point estimators $\hat\vartheta_n=\vartheta_0+O_{\P}(\delta_n)$. The proof is given in the appendix.

\begin{theo}\label{theoregr} Let $m$ fulfill the assumptions described at the beginning of Section \ref{sec-m}, and assumptions \eqref{ass1} and \eqref{ass2} from Subsection \ref{method4} hold.
Then, the regression estimator $\hat m_n(x)$ is uniformly consistent for $x\in [0,\hat\vartheta_n]\cup [\hat\vartheta_n+r_n,1]$ for all $r_n\downarrow0$ with $d_n=o(r_n)$. In detail,
\begin{align*}
\sup_{x\in[0,\hat\vartheta_n]\cup [\hat\vartheta_n+r_n,1]}|\hat m_n(x)-m(x)|=&\ O_\mathbb{P}(d_n^{q\wedge 1}r_n^{-(q\wedge 1)})
+O_{\P}(\delta_n^q)
+O_{\P}\Big(\sqrt{\frac{\log(n)}{nd_n}}\Big),
\end{align*}
where $\delta_n$ is the rate of $\hat\vartheta_n$, $|\hat\vartheta_n-\vartheta_0|=O_{\P}(\delta_n)$.
\end{theo}

\begin{cor}
With the same idea as in the proof of Lemma \ref{lem-deriv-pos} and Lemma \ref{lem-deriv} we get 
\begin{align*}
\int_0^1 |\hat m_n(x)-m(x)|dx=&\ O_\mathbb{P}(d_n^{q\wedge 1}r_n^{-(q\wedge1)})
+O_{\P}(\delta_n^q)
+O_{\P}\Big(\sqrt{\frac{\log(n)}{nd_n}}\Big)+\int_{\hat\vartheta_n}^{\hat\vartheta_n+r_n}\!\!|\hat m(x)-m(x)|dx\\
=&\ O_\mathbb{P}(d_n^{q\wedge 1}r_n^{-(q\wedge 1)})
+O_{\P}(\delta_n^q)
+O_{\P}\Big(\sqrt{\frac{\log(n)}{nd_n}}\Big) + O_\P(r_n\log(n))
\end{align*}
and
\begin{align*}
\int_0^1 (\hat m_n(x)-m(x))^2dx=&\ O_\mathbb{P}(d_n^{2(q\wedge 1)}r_n^{-2(q\wedge 1)})
+O_{\P}(\delta_n^{2q})
+O_{\P}\Big(\frac{\log(n)}{nd_n}\Big)+\int_{\hat\vartheta_n}^{\hat\vartheta_n+r_n}\!\!(\hat m(x)-m(x))^2dx\\
=&\ O_\mathbb{P}(d_n^{2(q\wedge 1)}r_n^{-2(q\wedge 1)})
+O_{\P}(\delta_n^{2q})
+O_{\P}\Big(\frac{\log(n)}{nd_n}\Big) + O_\P(r_n\log(n))
\end{align*}
for all $r_n\downarrow 0$ with $d_n=o(r_n)$. Choosing $r_n=d_n^{\frac{q\wedge 1}{1+(q\wedge 1)}}\log(n)^{-\frac{1}{1+(q\wedge 1)}}$, respectively $r_n=d_n^{\frac{2(q\wedge 1)}{1+2(q\wedge 1)}}\log(n)^{-\frac{1}{1+2(q\wedge 1)}}$, it thus holds
\begin{align*}
\int_0^1 |\hat m_n(x)-m(x)|dx =&\ O_\mathbb{P}(d_n^{\frac{q\wedge 1}{1+(q\wedge 1)}}\log(n)^{\frac{q\wedge 1}{1+(q\wedge 1)}})
+O_{\P}(\delta_n^q)
+O_{\P}\Big(\sqrt{\frac{\log(n)}{nd_n}}\Big)\\
\int_0^1 (\hat m_n(x)-m(x))^2dx=&\ O_\mathbb{P}(d_n^{\frac{2(q\wedge 1)}{1+2(q\wedge 1)}}\log(n)^{\frac{2(q\wedge 1)}{1+2(q\wedge 1)}})
+O_{\P}(\delta_n^{2q})
+O_{\P}\Big(\frac{\log(n)}{nd_n}\Big).
\end{align*}
\end{cor}

\begin{cor}\label{cor:regression_fixed}
For fixed $x$ the estimator $\hat m_n$ is consistent.
For fixed $x\in[0,1]\setminus \{\vartheta_0\}$ it holds
$$|\hat m_n(x)-m(x)|= O_\mathbb{P}(d_n^{q\wedge 1})
+O_{\P}(\delta_n^q)
+O_{\P}\Big(\sqrt{\frac{1}{nd_n}}\Big),$$
which follows with the same arguments as in the proof of Theorem \ref{theoregr} and Lemma \ref{lem-deriv-help}. 
For $x=\vartheta_0$ we obtain 
$$\hat m_n(x)-m(x)=\hat m_n(\vartheta_0)=\left(O_\P(1)+O_\P\Big(\frac{1}{n^{1/2}d_n^{q+1/2}}\Big)+O_\mathbb{P}\Big(\frac{\delta_n^{q\wedge 1}}{d_n^{2q}}\Big)\right)O_\P(\delta_n^q).$$
The proof for the last line is given in the appendix.
The bandwidth $d_n$ can be chosen such that the rate in the last line is always $o(1)$, so $\hat m_n$ is a consistent regression estimator. Note to this end, that $\delta_n$ does not depend on $d_n$ for methods (1)-(3). For method (4) and (5) there is also a bandwidth in the estimation of $M_n$, and thus in the rate of $\delta_n$, which we call $\tilde d_n$ for now, but this bandwidth $\tilde d_n$ does not have to be the same as the bandwidth $d_n$ used for the estimation of $\hat m_n$. Therefore, choose $\tilde d_n$ as the optimal bandwidth for estimating $M_n$ and then choose $d_n$ appropriately such that $\hat m_n$ is a consistent estimator for $m$.
\end{cor}

\section{Simulations}\label{sec-bootstrap}

Although the considered estimators $\hat\vartheta_n$ are consistent, they have a bias effect, and simulation results show that the estimators are too large. This is often the case in change point estimation because for detection of the change some distance to the change point is needed. We consider a bootstrap bias approximation to obtain better estimation results. 
Based on the original sample $(X_i,Y_i)$, $i=1,\dots,n$, we estimate $m$ by some consistent estimator $\hat m_n$ and the change point $\vartheta_0$ by some consistent $\hat\vartheta_n$. For each $b\in\{1,\dots,B\}$ create a residual bootstrap sample $(X_i,Y_i^*)$, $i=1,\dots,n$, with a regression change in $\hat\vartheta_n$: 
$$Y_i^*=\left\{\begin{array}{l} \hat\varepsilon_i^*\mbox{, if } X_i\leq\hat\vartheta_n\\ \hat m_n(X_i)+\hat\varepsilon_i^*\mbox{, if } X_i>\hat\vartheta_n\end{array}\right.$$
based on bootstrap residuals $\hat\varepsilon_1^*,\dots,\hat\varepsilon_n^*$ drawn with replacement from $\hat\varepsilon_i=Y_i-\hat m_n(X_i)$, $i=1,\dots,n$. A similar bootstrap procedure is e.\,g.\ used in \cite{AntochEtal} for estimating a change in the mean of $X_1,\ldots,X_n$.
We consider now a homoscedastic model; in heteroscedastic models wild bootstrap is more suitable. From the bootstrap data estimate the change point by $\hat\vartheta_{n,b}^*$. Then we approximate the bias by the median or mean of $\hat\vartheta_{n,b}^*-\hat\vartheta_n$, $b=1,\dots,B$, say $\hat{bias}$, and 
and replace $\hat\vartheta_n$ by $\hat\vartheta_n-\hat{bias}$.
The following graphics show the example with $n=400$, $B=500$, uniformly distributed covariates, normally distributed errors with $\sigma=0.2$. In the first example demonstrated in Figure \ref{fig:bootstrap1} the regression function is rather flat, $m(x)=0.5(x-0.4)_+$, and we used   $t_n=\sigma/n^{0.49}$. In the second example demonstrated in Figure \ref{fig:bootstrap2} the regression function is $m(x)=2(x-0.4)_+$ and thus the change point easier to estimate. We used $t_n=0.5\sigma/n^{0.49}$. 
In both cases the change point is $\vartheta_0=0.4$ and we applied method \ref{sec-Mn-1}. To generate the residuals for the bootstrap data we applied a local linear regression estimator $\hat m_n$ with cross-validation bandwidth choice.  The left graphics are one sample example with the true regression function. The other graphics are based on 100 replications. In the middle graphic the left boxplot is based on the 100 estimators $\hat \vartheta_n$, the middle boxplot on the median-based bias reduction and the right boxplot on the mean-based bias reduction. Those estimators are closer to the true change point. The right graphics demonstrate estimated densities of the estimators. The density ``1'' is based on the original estimators $\hat\vartheta_n$, ``2'' on the median-based bias reduction, and ``3'' on the mean-based bias reduction. 
In the first example demonstrated in Figure \ref{fig:bootstrap1} the median (mean, variance) of the 100 estimators are  0.6840 (0.6850, 0.004) for the original estimator, 0.5600  (0.5515, 0.007) for the median-based bias reduction, and 0.5563  (0.5535, 0.007) for the mean-based bias reduction. In the second example demonstrated in Figure \ref{fig:bootstrap2} the median (mean, variance) of the 100 estimators are  0.4685 (0.4375, 0.01) for the original estimator, 0.4255 (0.3854, 0.016) for the median-based bias reduction, and 0.45544 (0.41293, 0.018) for the mean-based bias reduction. 

\begin{figure}
    \centering
    \includegraphics[width=15cm]{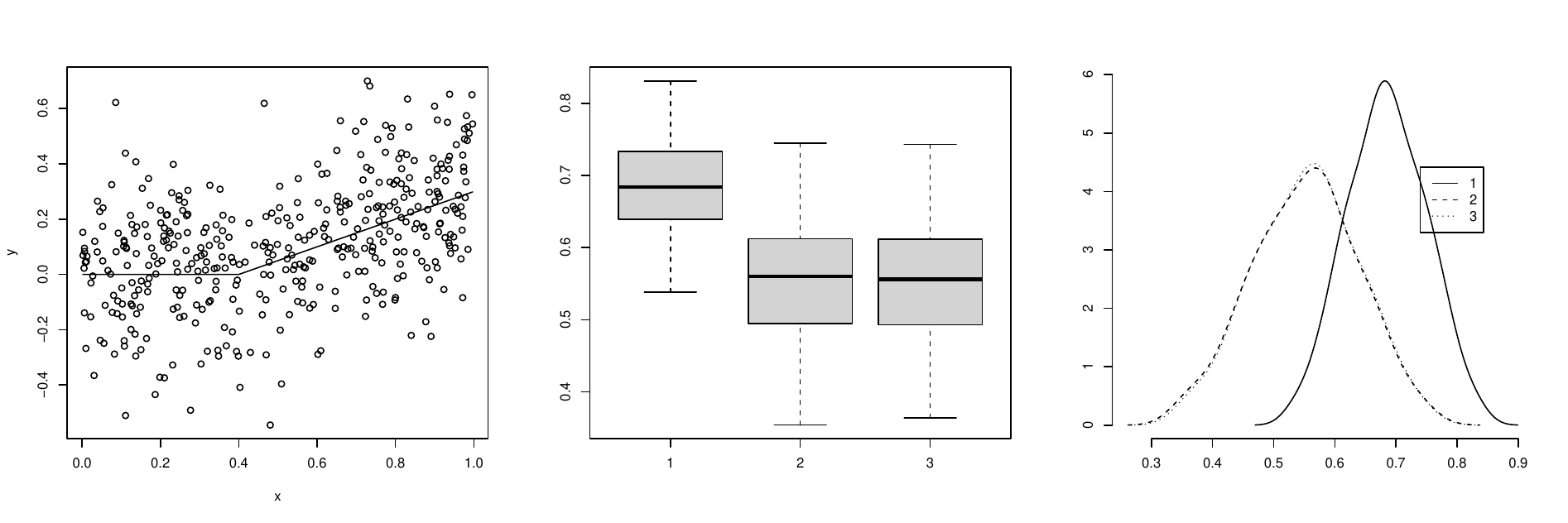}
    \caption{
    \it First example; left: scatterplot of observations and true regression function; middle: boxplots (left $\hat\vartheta_n$, middle median-based bias reduction estimator, right mean-based bias reduction restimator), right: 1 density $\hat\vartheta_n$, 2 density median-based bias reduction estimator, 3 mean-based bias reduction estimator.}
    \label{fig:bootstrap1}
\end{figure}

\begin{figure}
    \centering
    \includegraphics[width=15cm]{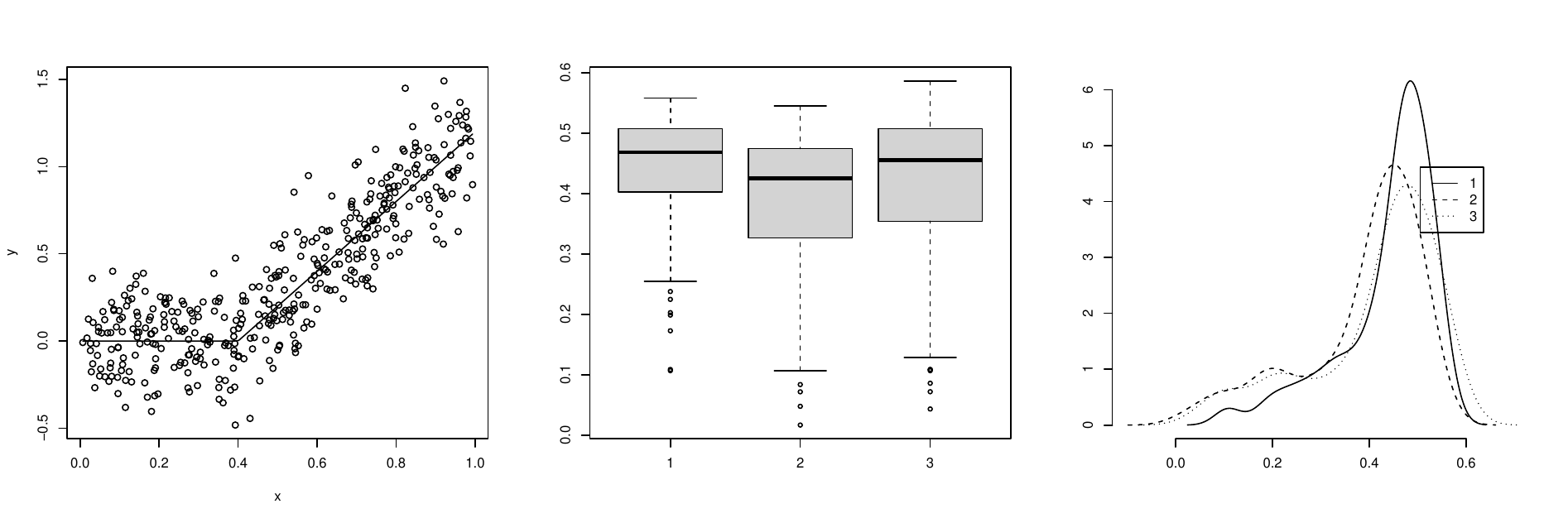}
    \caption{
    \it Second example; left: scatterplot of observations and true regression function; middle: boxplots (left $\hat\vartheta_n$, middle median-based bias reduction estimator, right mean-based bias reduction restimator), right: 1 density $\hat\vartheta_n$, 2 density median-based bias reduction estimator, 3 mean-based bias reduction estimator.}
    \label{fig:bootstrap2}
\end{figure}

To compare the five different methods to estimate the change point proposed in Section \ref{sec:methods} we simulate data with $m(x)=(x-0.4)_+$, $X\sim \mathcal{U}[0,1]$, $\varepsilon\sim\mathcal N(0,0.2^2)$ and apply all five methods. The results for sample size $n=100$ and $n=200$ are shown in Figure \ref{fig:compMeth}. All results are based on 300 independent repetitions and the bias reduction on 300 bootstrap replications. Since the results with median-based bias reduction and mean-based bias reduction are very similar in this setup, we only present the results without bias reduction and with median-based bias reduction to improve clarity.
The detailed simulation setup is as follows. 
The threshold $t_n$ is chosen as $t_n=n^{0.01}\cdot \text{'rate'}\cdot \text{'sd'}$, where 'rate' is the rate of $\hat M_n(\vartheta)-M(\vartheta)$ (e.\,g.\ $\text{'rate'}=n^{-1/2}$ for method (1) and (2)) and 'sd' is the empirical standard deviation of the data that $M_n$ is based on (e.\,g.\ $(Y_1,\ldots,Y_n)$ for method (1) and $(\hat m_n(X_1)^2,\ldots,\hat m_n(X_n)^2)$ for method (2)). To determine $\hat m_n$ we use a local linear kernel estimator with automatic bandwidth selection. The bandwidth $d_n$ in \eqref{hat-g} used for method (4) and (5) is chosen by a rule of thumb as $d_n=1.06\cdot n^{-1/2}\cdot \text{'sd$(Y)$'}$ for method (4) and $d_n=1.06\cdot n^{-3/5}\cdot \text{'sd$(Y)$'}$ for method (5), where 'sd$(Y)$' is the empirical standard deviation of $(Y_1,\ldots,Y_n)$.  

\begin{figure}
    \centering
    \includegraphics[width=0.9\textwidth]{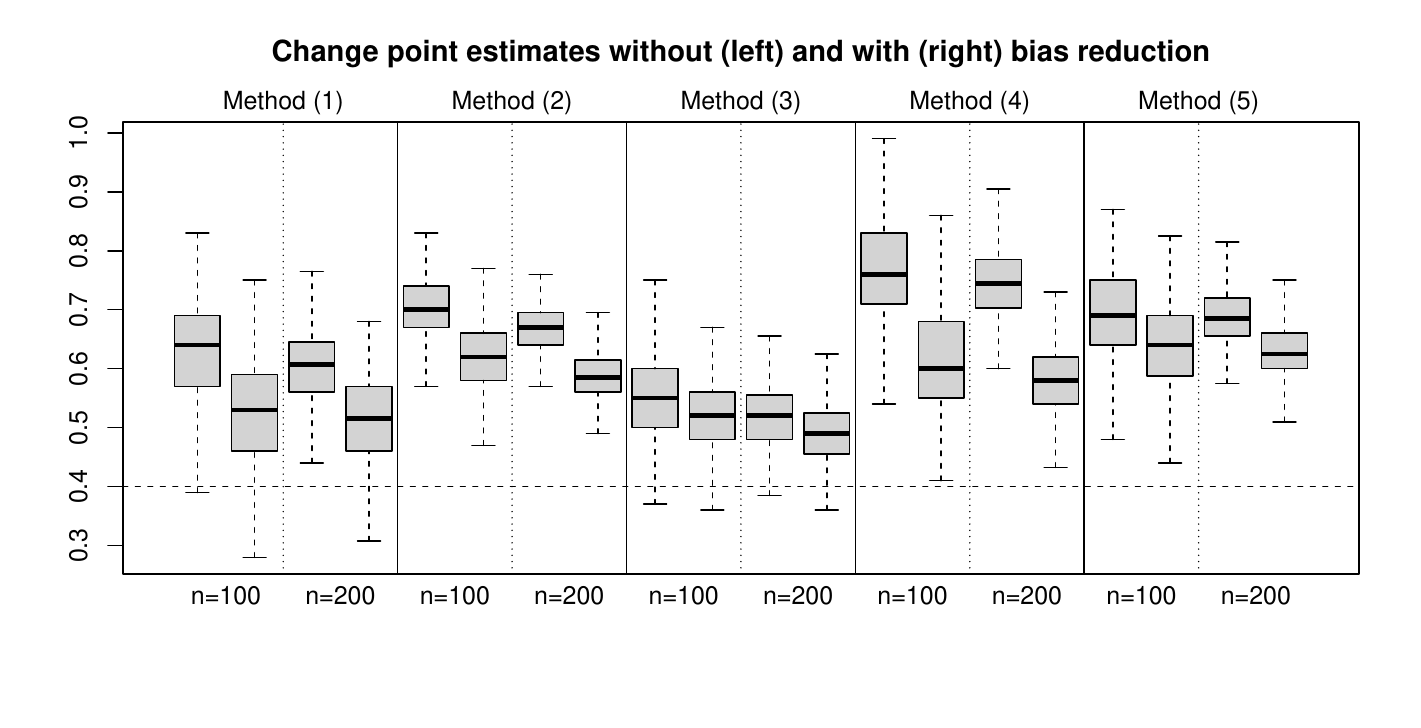}
    \caption{\it Results of change point estimators in 300 independent repetitions. The horizontal dashed line marks the true change point $\vartheta_0=0.4$. }
    \label{fig:compMeth}
\end{figure}

It is obvious that the size of the threshold $t_n$ has a high influence on the estimation performance. To illustrate this, we give estimation results for method (1), (2) and (3) with $t_n=c\cdot n^{0.01}\cdot \text{'rate'}\cdot \text{'sd'}$, $n=200$, for different values of $c$ in Figure \ref{fig:compTresh}.

\begin{figure}
    \centering
    \includegraphics[width=\textwidth]{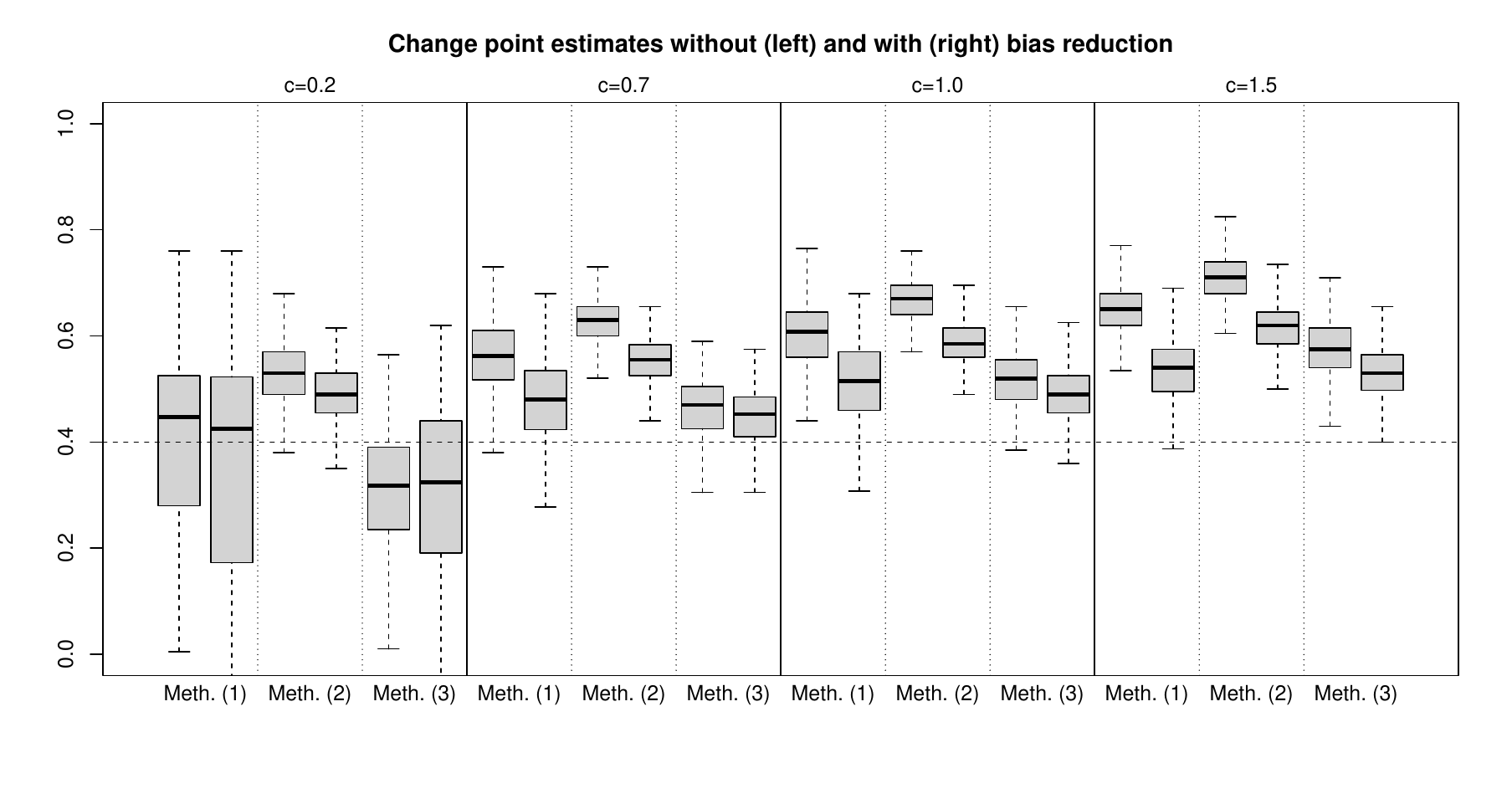}
    \caption{\it Results of change point estimators in 300 independent repetitions with different threshold $t_n$. The horizontal dashed line marks the true change point $\vartheta_0=0.4$.}
    \label{fig:compTresh}
\end{figure}

Likewise, the form of the true regression function influences the estimation performance, as can be seen in Figure \ref{fig:compRegr}. Results from data simulated with $m(x)=(x-0.4)_+^q$, $n=200$, for different values of $q$ are shown. The smaller $q$ is, the more pronounced is the change point and thus the easier the change point can be estimated.

\begin{figure}
    \centering
    \includegraphics[width=0.8\textwidth]{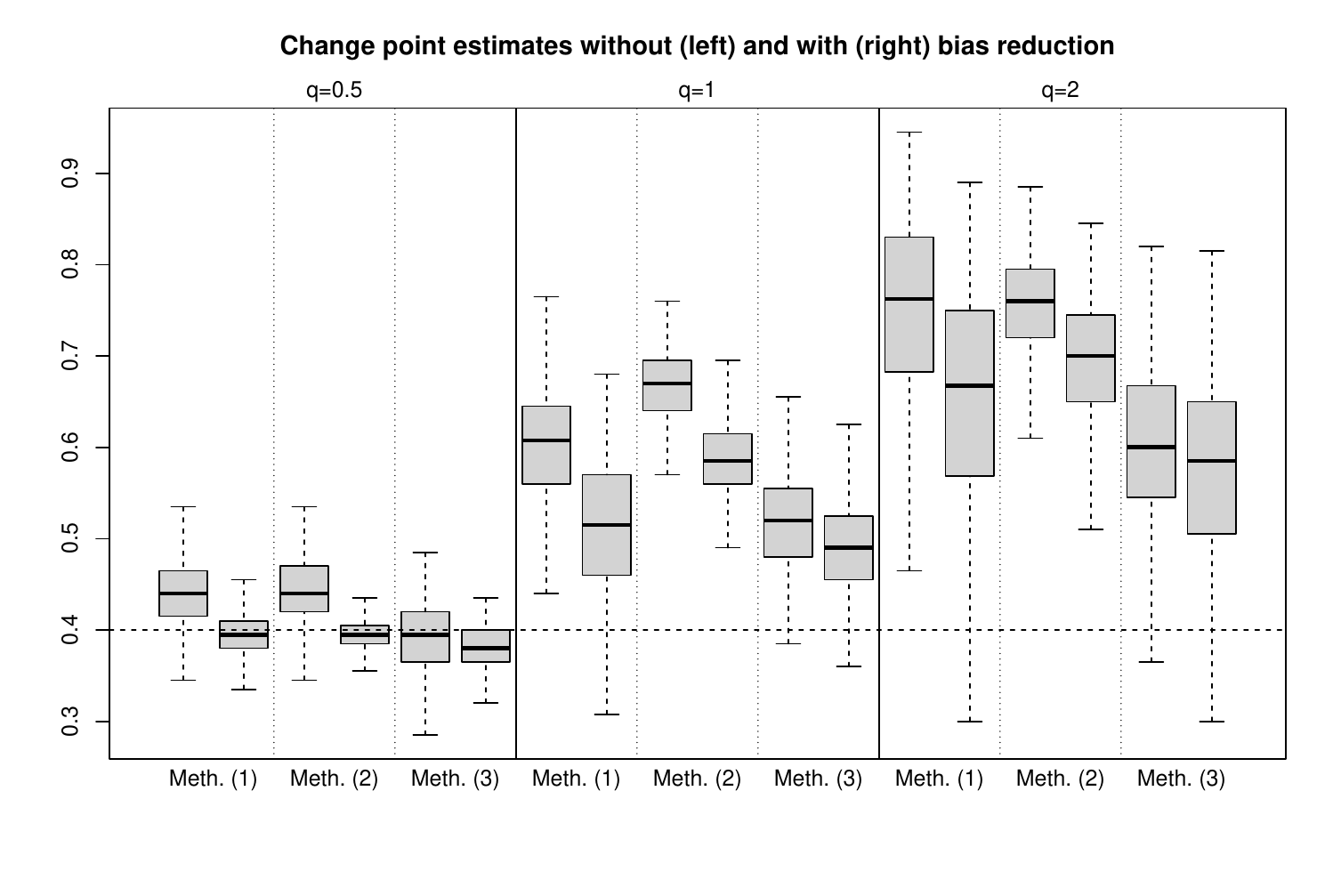}
    \caption{\it Results of change point estimators in 300 independent repetitions with different regression function $m$. The horizontal dashed line marks the true change point $\vartheta_0=0.4$.}
    \label{fig:compRegr}
\end{figure}

\subsection{Regression estimation}

In this subsection we show some simulation results for $\hat m_n$ as proposed in Section \ref{chap:regression}. The simulation setup is again $m(x)=(x-0.4)_+$, $X\sim \mathcal{U}[0,1]$, $\varepsilon\sim\mathcal{N}(0,0.2^2)$, $n=200$. Since $\hat m_n$ is based on $\hat\vartheta_n$ we first estimate this value by method (3) with median-based bias reduction based on 300 bootstrap replications. The estimation of $\hat\vartheta_n$ is repeated for each run. In Figure \ref{fig:regression} the results of four independent runs are presented. Our estimator $\hat m_n$ is compared to classical nonparametric regression estimators. As expected the classical estimators have some problems in the area where $m\equiv 0$. Additionally, the typical boundary effect of the classical kernel estimators is visible, where the local linear version  performs better than  the local constant one near 0 and 1. Our estimator $\hat m_n$ also shows this effect near 1, which is even worse due to the one-sided kernel. In the extreme case $x=1$ with high probability not a single observation can be used for the estimation. Therefore, we restrict the estimation of $\hat m_n$ to the interval $[0,1-d_n/2]$. 

\begin{figure}
    \centering
    \includegraphics[width=\linewidth]{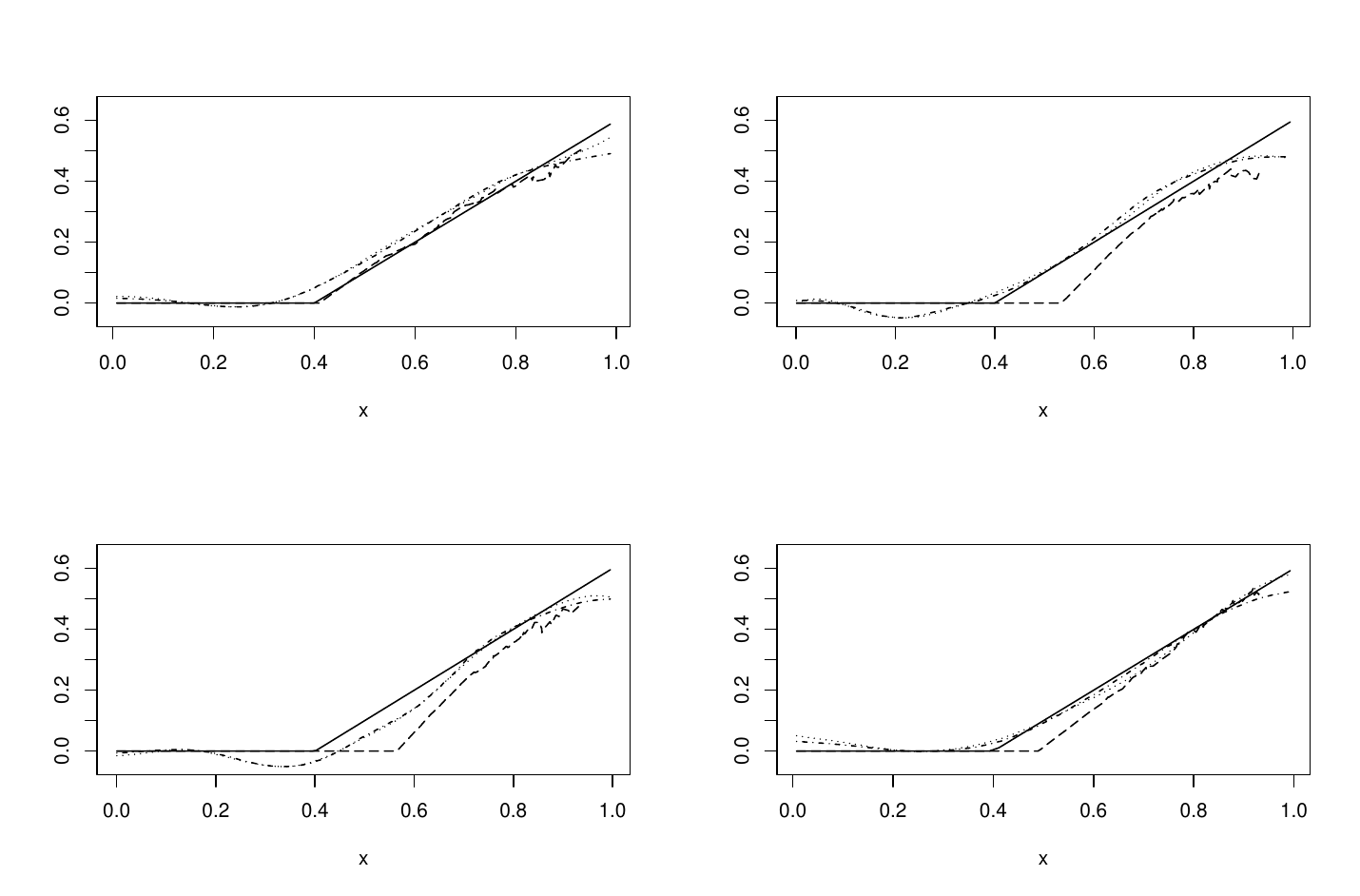}
    \caption{\it Estimation results for $\hat m_n$ as proposed in Section \ref{chap:regression} for four independent runs. The solid line is the true regression function. The dashed line is our estimator $\hat m_n$ (which depends on $\hat\vartheta_n$) where $\hat\vartheta_n$ is estimated by method (3) with bias reduction based on 300 bootstrap repetitions. As competitors we compare to a classic local linear (dotted line) and a local constant (dashed-dotted line) regression estimator with automatic bandwidth selection.  }
    \label{fig:regression}
\end{figure}

\section{Application in two sample models and real data analysis}\label{sec-application}

We consider two independent regression models
\beq
Y_{1,i}&=&m_1(X_{1,i})+\e_{1,i},\; i=1,\dots,n_1,\\
Y_{2,i}&=&m_2(X_{2,i})+\e_{2,i},\; i=1,\dots,n_2,
\eeq
with  design points in $[0,1]$. In this section we consider random design, but as discussed in Section \ref{fixed-design} it can be modified for the fixed design case. The two regression functions are assumed to be equal on some area but differ on some other, i.\,e.\ we assume that there exists some $\vartheta_0\in(0,1)$ and some $\Delta>0$ such that
\beq
m_1(x)&=&m_2(x)\quad \forall  x\in[0,\vartheta_0]\\
m_1(x)&\neq& m_2(x)\quad \forall x\in(\vartheta_0,\vartheta_0+\Delta).
\eeq
A typical example where such data occurs is the growth of children where $(X_1,Y_1)$ models the age and height of boys and $(X_2,Y_2)$ the age and height of girls. For younger children the relation between age and height is the same for both sexes but from some age on it is different. We consider a real data example showing this effect below. The same data set has been considered in \cite{Hlavka-Huskova} using parametric, heteroscedastic regression models to test for a change in the jumping speed - age relation. 

We are interested in estimating the change point $\vartheta_0$. In the case of the same covariates $X_i=X_{1,i}=X_{2,i}$ for all $i= 1,\dots,n=n_1=n_2$ one can define
\beq
Z_i:=Y_{1,i}-Y_{2,i}&=&(m_1-m_2)(X_i)+\varepsilon_{i,1}-\varepsilon_{i,2}\\
&=:&m(X_i)+\varepsilon_i
\eeq
where $m=m_1-m_2$ is assumed to be a continuous function that fulfills $m(x)=0$ for $x\in [0,\vartheta_0]$ and $m(x)\neq 0$ for $x\in(\vartheta_0,\vartheta_0+\Delta)$. Then methods (1)--(5) from Subsections \ref{method1}--\ref{method5} can be applied directly. 

In the case of different covariate values we now use the notations $M_{n_1,n_2}$, $\hat\vartheta_{n_1,n_2}$ and $t_{n_1,n_2}$ instead of $M_n$, $\hat\vartheta_n$ and $t_n$. One can still define a difference regression function $m=m_1-m_2$. Under the assumption that $m$ is zero on $[0,\vartheta_0]$, nondecreasing on $[0,1]$, and strictly increasing on $(\vartheta_0,\vartheta_0+\Delta)$ one can analogously as in Subsection \ref{method3} apply $$M_{n_1,n_2}(\vartheta)=\hat m_{1,n_1}(\vartheta)-\hat m_{2,n_2}(\vartheta)\mbox{ and } M(\vartheta)=m(\vartheta)$$ with nonparametric regression functions $\hat m_{1,n_1}$ in the first sample and $\hat m_{2,n_2}$ in the second sample (method (3a)). 

Similar to Subsection \ref{method2} one can use 
$$M_{n_1,n_2}(\vartheta)=\int_0^\vartheta (\hat m_{1,n_1}(x)-\hat m_{2,n_2}(x))^2\,d x \mbox{ and } M(\vartheta)=\int_{\vartheta_0}^\vartheta m^2(x)\,dxI\{\vartheta>\vartheta_0\}$$
(method (2a)).
As in Subsection \ref{method1} one can apply 
$$M_{n_1,n_2}(\vartheta)=\frac{1}{n_1}\sum_{i=1}^{n_1} Y_{1,i}I\{X_{1,i}\leq\vartheta\}-\frac{1}{n_2}\sum_{i=1}^{n_2} Y_{2,i}I\{X_{2,i}\leq\vartheta\},$$
which estimates
$$M(\vartheta)=\int_{0}^\vartheta m_1(x)\, dF_{X_1}(x)-\int_{0}^\vartheta m_2(x)\, dF_{X_2}(x)=\int_{\vartheta_0}^\vartheta m(x)\, dF_X(x)I\{\vartheta>\vartheta_0\}$$
(method (1a)), where the last equality holds if the covariate distributions are the same $F_X$ in both samples. If both covariate distributions differ, method (1a) possibly detects a difference in the distribution rather than in the regression function. Assume for this sake that $f_1$ is the density od $X_1$, $f_2$ the density of $X_2$ and that there  exists some $\vartheta_F<\vartheta_0$ with $f_1(x)=f_2(x)$ for $x\in[0,\vartheta_F]$ and $f_1(x)>f_2(x)$ for $x\in(\vartheta_F,\vartheta_F+\Delta)$.
Then method (1a) would detect $\vartheta_F$.

We applied method (1a), (2a) and (3a) on a real data set containing the exact age and height of 781 children and adolescents (426 girls and 355 boys) aged 6-19 years. The data was collected in a study on muscle function assessed using jumping mechanography with pupils from Prague and \'{U}pice, a small city in eastern Czech Republic. For a thorough analysis of the data set see \cite{SumnikEtal2013}. In Figure \ref{fig:BoysGirls} the relation between the age and the height for girls vs boys is shown. It can be seen, that from the age of approximately 13 on the estimated regression function differs for both sexes. The results from our estimation procedures are shown in Table \ref{tab:BoysGirls}. 
We scaled the $X$-values on $[0,1]$ to fit our model assumptions and after calculating $\hat\vartheta_{n_1,n_2}$ rescaled to interpret the change point $\hat\vartheta_{n_1,n_2}(Age)$. As the data appears to contain heteroscedasticity, we use wild bootstrap for the bias reduction. In the same way as described for the simulations in Section \ref{sec-bootstrap} we chose $t_{n_1,n_2}=c\cdot \sqrt{(n_1^{0.01}\cdot \text{'rate'}_{n_1}\cdot\text{'sd'}_1)^2+(n_2^{0.01}\cdot \text{'rate'}_{n_2}\cdot\text{'sd'}_2)^2 }$ with $c=0.8$ and calculated $\hat m_{1,n_1}$ and $\hat m_{2,n_2}$ with local linear kernel estimators. 

Since the covariates $X_{1,i}$ (age of boys) and $X_{2,i}$ (age of girls) are not equal, methods (1)-(5) can not be applied to this data set. The covariates also do not have the same distribution, see Figure \ref{fig:BoysGirlsDensity}. As explained, method (1a) therefore presumably detects the change in the difference of distributions.

\cite{GasserEtal1984} investigate the modeling of growth curves for individual children and adolescents. They propose a nonparametric method and compare it to classic parametric approaches to model growth curves. They also identified the advent of puberty around the age of 12 for girls and 14 for boys as the main growth spurt, which corresponds with our change point estimates.

\begin{table}[h]
    \centering
    \begin{tabular}{c|c|c|c}
        & Method (1a) & Method (2a) & Method (3a) \\
        \hline
       $\hat\vartheta_{n_1,n_2}$  & 0.253 & 0.769 & 0.708 \\
       $\hat\vartheta_{n_1,n_2}(Age)$ & 9.33 & 16.01 & 15.22 \\
       $\hat\vartheta_{n_1,n_2}-\hat{bias}_{median}$ & 0.198 & 0.672  & 0.566\\
       $\big(\hat\vartheta_{n_1,n_2}-\hat{bias}_{median}\big)(Age)$ & 8.63  &  14.76 & 13.38 \\
       $\hat\vartheta_{n_1,n_2}-\hat{bias}_{mean}$ & 0.201  &  0.670 & 0.558 \\
       $\big(\hat\vartheta_{n_1,n_2}-\hat{bias}_{mean}\big)(Age)$ & 8.66 &  14.73 &  13.27
    \end{tabular}
    \caption{\it Estimation results for the change point in the height-age relation of girls vs boys}
    \label{tab:BoysGirls}
\end{table}

\begin{figure}
    \centering 
\includegraphics[width=0.8\textwidth]{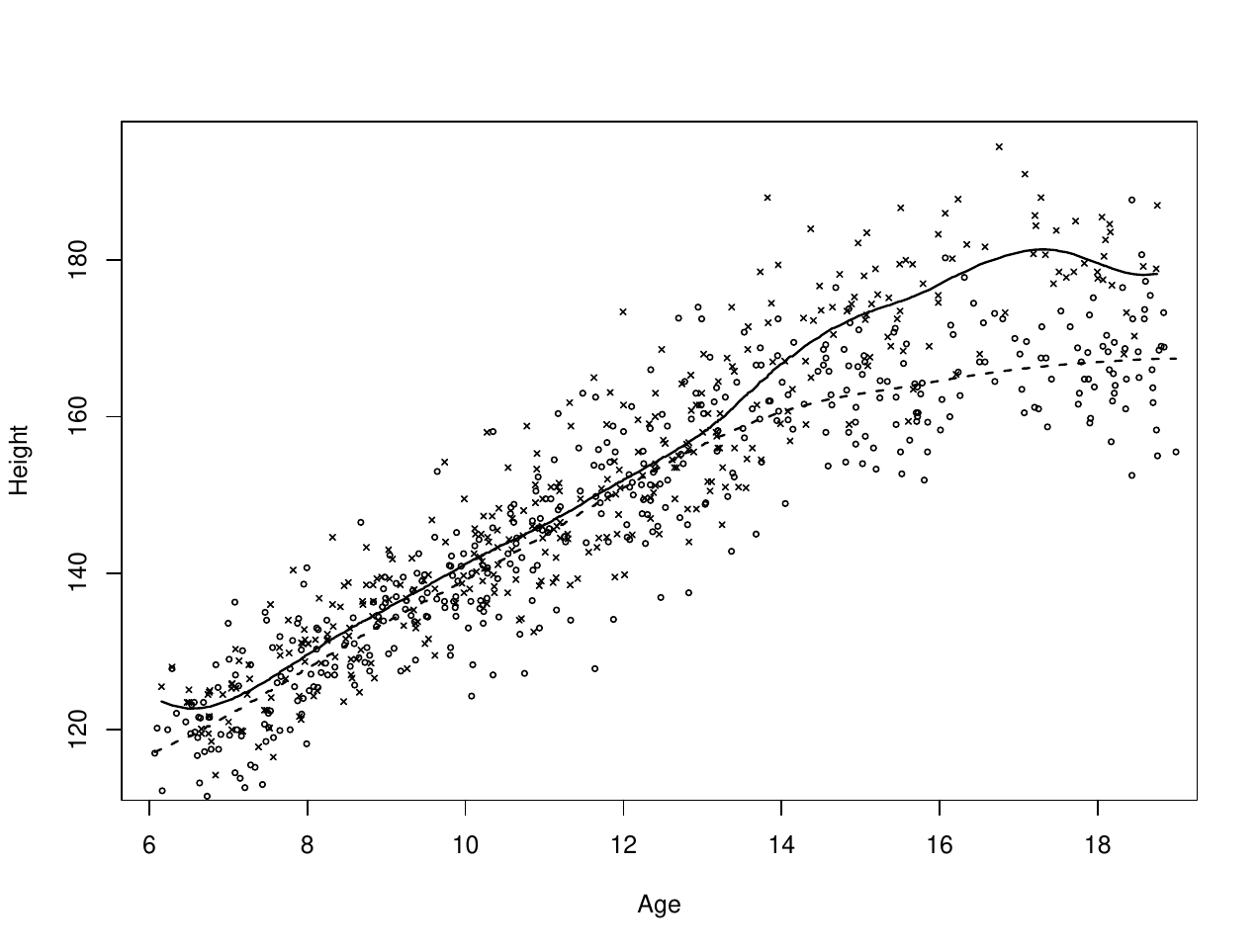}
\caption{\it The graph shows the age and height of 781 pupils, as well as the estimated regression function (local linear estimator). The circles and dashed line represent data from girls, and the crosses and solid line data from boys.}
    \label{fig:BoysGirls}
\end{figure}

\begin{figure}
    \centering 
\includegraphics[width=0.6\textwidth]{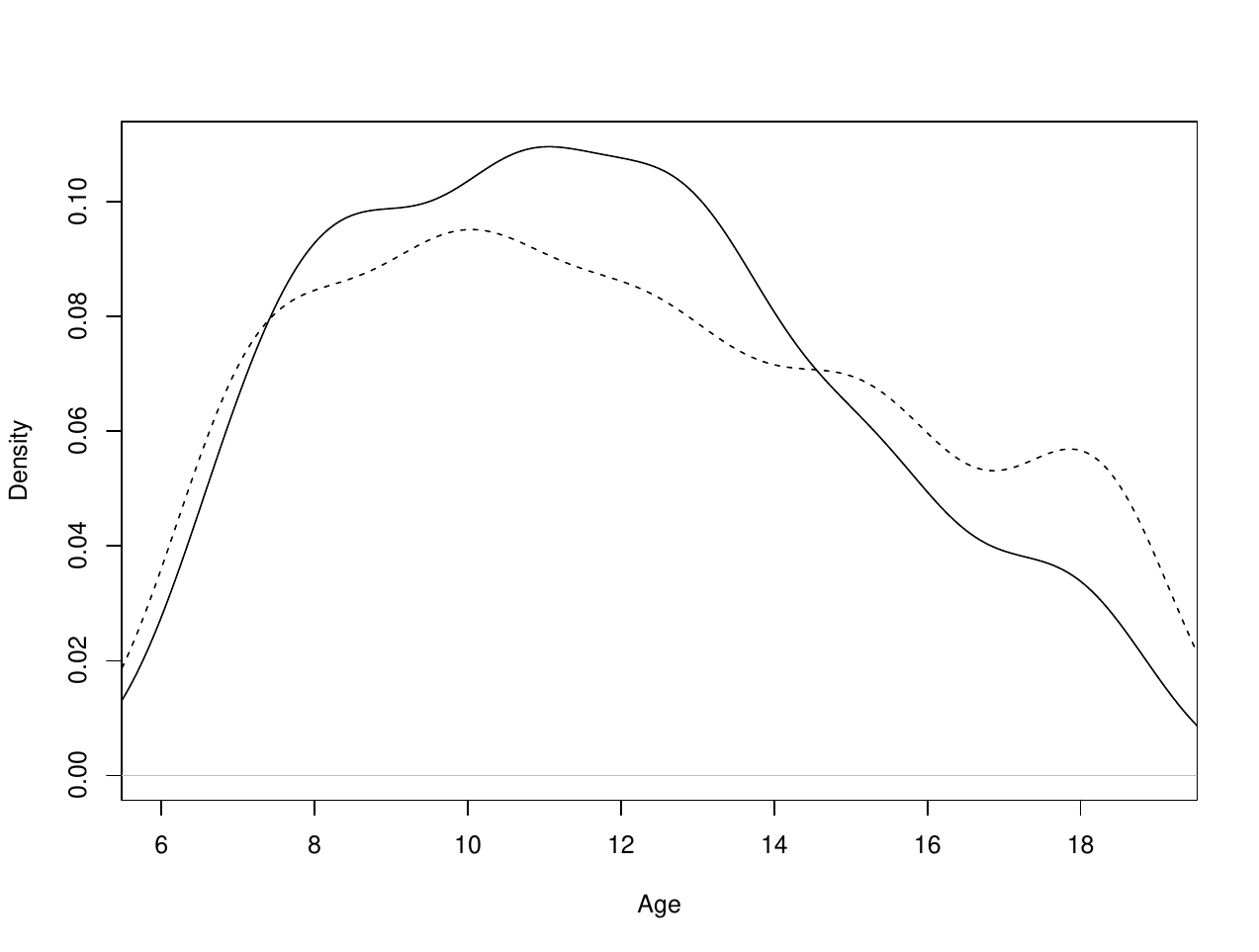}
\caption{\it The graph shows the estimated density of the age of girls (dashed line) and boys (solid line).}
    \label{fig:BoysGirlsDensity}
\end{figure}

\section{Concluding remarks}

In this article we suggested a general method to estimate the maximum of an argmin of some objective function in a case where  classical argmin procedures do not work. We applied it for estimating gradual changes of nonparametric regression functions for an independent sample. But Theorem \ref{theo1} could be applied under much more general assumptions on the regression model like for time series data, and also for very different estimation problems. For the suggested methods we showed consistency and rates of convergence. The next step will be deriving asymptotic distributions of the change point estimators. 
To estimate smooth changes in a nonparametric regression function one could also use an alternative method by checking for a jump in the derivative in particular in the case considered in Section \ref{sec-deriv}. 
Jumps in the derivative have been considered by \cite{Raimondo}, \cite{Gijbels-Goderniaux}, among others. 
We will consider a different method in a future work related to \cite{Antoch-etal}.

\begin{appendix}

\section{Proofs}

\subsection{Proof of Theorem \ref{theo1}}

Because $M$ is zero on $[0,\vartheta_0]$ and strictly positive on $(\vartheta_0,1]$ we obtain  
\begin{eqnarray*}
|\hat\vartheta_n-\vartheta_0|&=& \left|\int_0^1(I\{ M_n(\vartheta)\leq t_n\}-I\{\vartheta\in [0,\vartheta_0]\})\,d\vartheta\right|\\
&=& \Big|\int_0^{\vartheta_0}(I\{M_n(\vartheta)\leq t_n\}-1)\,d\vartheta
+\int_{\vartheta_0}^1 I\{M_n(\vartheta)\leq t_n\}\,d\vartheta\Big|\\
&\leq & \int_0^{\vartheta_0}I\{M_n(\vartheta)-M(\vartheta)>t_n\}\,d\vartheta +  \int_0^1 I\{M_n(\vartheta)\leq t_n\} I\{ 0< M(\vartheta)\}\,d\vartheta.
\end{eqnarray*}
Now with $\delta_n$ defined in Theorem \ref{theo1} we obtain for each $K>0$ by applying (\ref{tn})
\begin{eqnarray*}
&&\P\left( |\hat\vartheta_n-\vartheta_0|>K\delta_n\right)\\
&\leq &
\P\left( |\hat\vartheta_n-\vartheta_0|>K\delta_n, \sup_{\vartheta\in[0,1]}|M_n(\vartheta)-M(\vartheta)|\leq t_n\right) +o(1)
\\
&\leq&
\P\left( \int_0^1 I\{M_n(\vartheta)\leq t_n\} I\{ 0< M(\vartheta)\}\,d\vartheta>K\delta_n, \sup_{\vartheta\in[0,1]}|M_n(\vartheta)-M(\vartheta)|\leq t_n\right)\\
&&{} +o(1)\\
&\leq& \P\left( \int_0^1  I\{ 0< M(\vartheta)\leq 2 t_n\}\,d\vartheta>K\delta_n\right) +o(1).
\end{eqnarray*}
The integral inside the probability in the last line is not random. Note that we assume that $M$ is non-decreasing on $[0,1]$, $M(\vartheta)=0$ on $[0,\vartheta_0]$, and $M$ is strictly increasing on $(\vartheta_0,\vartheta_0+\Delta)$. Because $t_n$ converges to zero  we will have $2t_n<M(\vartheta_0+\Delta)$ for $n$ large enough, and the inequalities $0< M(\vartheta)\leq 2 t_n$ are equivalent to $\vartheta_0<\vartheta\leq M^{-1}(2t_n)$. Then the integral equals $M^{-1}(2t_n)-\vartheta_0=\delta_n$. Thus choosing $K\geq 1$ the probability is zero and we obtain the assertion. 
\hfill $\Box$

\subsection{Proof of Lemma \ref{lem-non-zero-regression}}

Here we use the notation $P=\mathbb{P}^X$, and $P_n$ for the empirical distribution of $X_1,\dots,X_n$. Then one can write
$$M_n(\vartheta)-M(\vartheta)= P_n g_\vartheta-Pg_\vartheta+P_nf_{\vartheta,\hat m_n-m}$$
with 
$$g_\vartheta(x)=m^2(x)I\{x\leq\vartheta\}$$
and
$$f_{\vartheta,h}(x)=h(x)(h(x)+2m(x))I\{x\leq\vartheta\}.$$
Then the function class $\mathcal{G}=\{ g_\vartheta\mid \vartheta\in [0,1]\}$ is a $P$-Donsker class by  Lemma 2.6.18(iv), \cite{vanderVaartWellner}, and thus
$$\sup_{\vartheta\in[0,1]}|P_ng_\vartheta-Pg_\vartheta|=O_{\mathbb{P}}(n^{-1/2}).$$
For the second term consider
$$P f_{\vartheta, h}=\int_0^\vartheta h(x)(h(x)+2m(x))\, dF_X(x).$$
Inserting $h=\hat m_n-m$ assumption (\ref{assumption1}) gives
$$ P f_{\vartheta, \hat m_n-m}=\int_0^\vartheta (\hat m_n^2(x)-m^2(x))\, dF_X(x)=O_{\mathbb{P}}(n^{-1/2})$$
uniformly in $\vartheta\in[0,1]$.
Then we obtain with assumption (\ref{assumption2})
$$\sup_{\vartheta\in[0,1]}|P_n f_{\vartheta, \hat m_n-m}|\leq \sup_{\vartheta\in[0,1],h\in \mathcal{H}} |P_n f_{\vartheta,h}-P f_{\vartheta,h}|+O_{\mathbb{P}}(n^{-1/2})=O_{\mathbb{P}}(n^{-1/2}).$$
The last step follows from the Donsker property in assumption (\ref{assumption2}) because  by Examples 2.5.4, 2.10.7, 2.10.8 in \cite{vanderVaartWellner} it follows that also 
 $\{f_{\vartheta,h}\mid \vartheta\in[0,1],h\in\mathcal{H}\}$ is a Donsker class. 
\hfill $\Box$

\subsection{Auxiliary Lemma to prove Lemma \ref{lem-deriv-pos}, Lemma \ref{lem-deriv} and Theorem \ref{theoregr}}
Before proving Lemma \ref{lem-deriv-pos} we state an auxiliary result that we will use to prove Lemma \ref{lem-deriv-pos}, Lemma \ref{lem-deriv} as well as Theorem \ref{theoregr}.
\begin{lemma}\label{lem-deriv-help}
Let $m(x)=g(x)(x-\vartheta_0)_+^q$ for some Lipschitz function $g:[0,1]\to\mathbb{R}_0^+$, and $\hat m_\vartheta(x)=\hat g_\vartheta(x)(x-\vartheta)_+^q$ with the kernel estimator $\hat g_\vartheta$ defined in (\ref{hat-g}), 
where the kernel $\kappa$ is one-sided with bounded support $[0,C_\kappa]$ and Lipschitz continuous. Further, let assumptions \eqref{ass1} and \eqref{ass2} from Subsection \ref{method4} hold.
Then for every $r_n\downarrow 0$ with $d_n=O(r_n)$ it holds
$$\sup_{\vartheta\in[0,1]}\sup_{x\in[\vartheta+r_n,1]}\big|\hat m_\vartheta(x)-m(x)\big|=O_{\mathbb{P}}(d_n^{q\wedge 1}r_n^{-(q\wedge 1)})+O_{\mathbb{P}}\Big(\sqrt{\frac{\log(n)}{nd_n}}\Big).$$
\end{lemma}
\textit{Proof:}
It holds
\begin{align}
\hat m_\vartheta&(x)-m(x)
= 
\frac{\frac 1{nd_n}\sum_{i=1}^n Y_i\kappa\big(\frac{X_i-x}{d_n}\big)(X_i-\vartheta)_+^q(x-\vartheta)_+^q}{\frac 1{nd_n}\sum_{i=1}^n \kappa\big(\frac{X_i-x}{d_n}\big)(X_i-\vartheta)_+^{2q}}-m(x)\nonumber\\
=&\ \frac{\frac 1{nd_n}\sum_{i=1}^n \varepsilon_i\kappa\big(\frac{X_i-x}{d_n}\big)(X_i-\vartheta)_+^q(x-\vartheta)_+^q}{\frac 1{nd_n}\sum_{i=1}^n \kappa\big(\frac{X_i-x}{d_n}\big)(X_i-\vartheta)_+^{2q}}\nonumber\\
&+\frac{\frac 1{nd_n}\sum_{i=1}^n (g(X_i)-g(x))\kappa\big(\frac{X_i-x}{d_n}\big)(X_i-\vartheta)_+^{q}(X_i-\vartheta_0)_+^{q}(x-\vartheta)_+^q}{\frac 1{nd_n}\sum_{i=1}^n \kappa\big(\frac{X_i-x}{d_n}\big)(X_i-\vartheta)_+^{2q}}\nonumber\\
&+ g(x) \frac{\frac 1{nd_n}\sum_{i=1}^n \kappa\big(\frac{X_i-x}{d_n}\big)(X_i-\vartheta)_+^q\big((X_i-\vartheta_0)_+^{q}(x-\vartheta)_+^q-(X_i-\vartheta)_+^q(x-\vartheta_0)_+^{q}\big)}{\frac 1{nd_n}\sum_{i=1}^n \kappa\big(\frac{X_i-x}{d_n}\big)(X_i-\vartheta)_+^{2q}}\nonumber,
\end{align}
which follows remembering that $Y_i=g(X_i)(X_i-\vartheta_0)_+^q+\varepsilon_i$ and $m(x)=g(x)(x-\vartheta_0)_+^q$.
The random denominator can be replaced by the deterministic $\frac 1{d_n}E\big[\kappa\big(\frac{X_1-x}{d_n}\big)\big](x-\vartheta)^{2q}$, since using that $\kappa$ is a one-sided kernel we get
\begin{align*}
&\inf_{x\in[0,1],\vartheta\in[0,1]:x>\vartheta}\bigg|\frac{\frac 1{nd_n}\sum_{i=1}^n\kappa\Big(\frac{X_i-x}{d_n}\Big)(X_i-\vartheta)^{2q}}{\frac 1{d_n}E\Big[\kappa\Big(\frac{X_1-x}{d_n}\Big)\Big](x-\vartheta)^{2q}}\bigg|\\
\geq&\ \inf_{x\in[0,1]}\bigg|\frac{\frac 1{nd_n}\sum_{i=1}^n\kappa\Big(\frac{X_i-x}{d_n}\Big)}{\frac 1{d_n}E\Big[\kappa\Big(\frac{X_1-x}{d_n}\Big)\Big]}\bigg|\\
\geq&\ 1-\frac{\sup_{x\in[0,1]}\Big|\frac 1{nd_n}\sum_{i=1}^n\kappa\Big(\frac{X_i-x}{d_n}\Big)-\frac 1{d_n}E\Big[\kappa\Big(\frac{X_1-x}{d_n}\Big)\Big]\Big|}{\inf_{x\in[0,1]}\frac 1{d_n}E\Big[\kappa\Big(\frac{X_1-x}{d_n}\Big)\Big]}\\
=&\ 1+o_{\P}(1).
\end{align*}
In the following we will only consider the denominator $(x-\vartheta)_+^{2q}$, since $\frac 1{d_n}E\big[\kappa\big(\frac{X_1-x}{d_n}\big)\big]=\int \kappa(u)f_X(ud_n+x)du\approx f_X(x)$ and $\inf_{x\in[0,1]}f_X(x)>0$.

For $q\in (0,1]$ the function $x\mapsto x^q$ is $q$-H\"older continuous with $|a^q-b^q|\leq |a-b|^q$ and for $q>1$ with the mean value theorem one obtains $|a^q-b^q|\leq q(|a|^{q-1}\vee |b|^{q-1})|a-b|$.

Thus, for $\vartheta<x\leq X_i$,
\begin{align}\nonumber
\Big(\frac{X_i-\vartheta}{x-\vartheta}\Big)^{q}-1&\leq (q\vee 1) \Big(\frac{X_i-\vartheta}{x-\vartheta}\Big)^{(q-1)\vee 0}\Big(\frac{X_i-x}{x-\vartheta}\Big)^{q\wedge 1}\\
\label{regression4}&\leq (q2^{q-2}\vee 1)\Big(1+\Big(\frac{X_i-x}{x-\vartheta}\Big)^{(q-1)\vee 1}\Big)\Big(\frac{X_i-x}{x-\vartheta}\Big)^{q\wedge 1}
\end{align}
using that $|a+b|^q\leq 2^{q-1}|a|^q+2^{q-1}|b|^q$ for all $a,b\in\mathbb{R}$, $q>1$, 
and that $(1+a)^q\leq 1+a$ for $q\in(0,1]$, $a>0$.

Further, with \eqref{regression4}, we get 
\begin{align*}
&\sup_{\vartheta\in[0,1],x\in[\vartheta+r_n,1]}\bigg|\frac{\frac 1{nd_n}\sum_{i=1}^n (g(X_i)-g(x))\kappa\big(\frac{X_i-x}{d_n}\big)(X_i-\vartheta)_+^{q}(X_i-\vartheta_0)_+^{q}(x-\vartheta)_+^q}{(x-\vartheta)_+^{2q}}\bigg|\\
\leq&\ \sup_{\vartheta\in[0,1],x\in[\vartheta+r_n,1]}\frac 1{nd_n}\sum_{i=1}^n \kappa\Big(\frac{X_i-x}{d_n}\Big)|g(X_i)-g(x)|\Big(\frac{X_i-\vartheta}{x-\vartheta}\Big)^{q}\\
=&\ O(d_n)\cdot O(1+d_n^{q\wedge 1}r_n^{-(q\wedge 1)})\cdot \sup_{x\in[0,1]}\frac 1{nd_n}\sum_{i=1}^n \kappa\Big(\frac{X_i-x}{d_n}\Big)\\
=&O_{\P}(d_n)
\end{align*}
since $\kappa(\frac{X_i-x}{d_n})=0$ if $X_i-x>C_\kappa d_n$ for some $C_\kappa<\infty$, $g$ is Lipschitz by assumption and $d_n=O(r_n)$.
Analogously
\begin{align*}
&\sup_{\vartheta\in[0,1],x\in[\vartheta+r_n,1]}\bigg|g(x) \frac{\frac 1{nd_n}\sum_{i=1}^n \kappa\big(\frac{X_i-x}{d_n}\big)(X_i-\vartheta)_+^q\big((X_i-\vartheta_0)_+^{q}(x-\vartheta)_+^q-(X_i-\vartheta)_+^q(x-\vartheta_0)_+^{q}\big)}{(x-\vartheta)_+^{2q}}\bigg|\\
\leq&\ \sup_{\vartheta\in[0,1],x\in[\vartheta+r_n,1]}\bigg|g(x) \frac{\frac 1{nd_n}\sum_{i=1}^n \kappa\big(\frac{X_i-x}{d_n}\big)(X_i-\vartheta)_+^q(x-\vartheta)_+^q\big((X_i-\vartheta_0)_+^{q}-(x-\vartheta_0)_+^{q}\big)}{(x-\vartheta)_+^{2q}}\bigg|\\
&+\sup_{\vartheta\in[0,1],x\in[\vartheta+r_n,1]}\bigg|g(x) \frac{\frac 1{nd_n}\sum_{i=1}^n \kappa\big(\frac{X_i-x}{d_n}\big)(X_i-\vartheta)_+^q(x-\vartheta_0)_+^q\big((x-\vartheta)_+^q-(X_i-\vartheta)_+^q\big)}{(x-\vartheta)_+^{2q}}\bigg|\\
\leq&\ O_{\P}(d_n^{q\wedge 1})\\
&+\sup_{x\in[0,1]}|g(x)| \sup_{\vartheta\in[0,1],x\in[\vartheta+r_n,1]}\frac 1{nd_n}\sum_{i=1}^n \kappa\Big(\frac{X_i-x}{d_n}\Big)\Big(\frac{X_i-\vartheta}{x-\vartheta}\Big)^{q}\Big(\Big(\frac{X_i-\vartheta}{x-\vartheta}\Big)^{q}-1\Big)\\
=&\ O_{\P}(d_n^{q\wedge 1})+O_{\P}(d_n^{q\wedge 1}r_n^{-(q\wedge 1)})= O_{\P}(d_n^{q\wedge 1}r_n^{-(q\wedge 1)}),
\end{align*}
where we again used equation \eqref{regression4}.

For the remaining term we show that
\begin{align}
\sup_{\vartheta\in[0,1],x\in[\vartheta+r_n,1]}\bigg|\frac{\frac 1{nd_n}\sum_{i=1}^n \varepsilon_i\kappa\big(\frac{X_i-x}{d_n}\big)(X_i-\vartheta)_+^q(x-\vartheta)_+^q}{(x-\vartheta)_+^{2q}}\bigg|=O_{\P}(n^{-1/2}d_n^{-1/2}\log(n)^{1/2}).\label{regression2}
\end{align}
Let $\vartheta_1,\ldots,\vartheta_{N(\epsilon)}$ be an $\epsilon$-net for $[0,1]$ with $\epsilon^{q\wedge 1}=n^{-\frac 12}d_n^{\frac 32}r_n^{q\wedge 1}$. 
For all $\vartheta\in[0,1]$ set $\vartheta_{j(\vartheta)}=\arg\min_{\vartheta'\in\{\vartheta_1,\ldots,\vartheta_{N(\epsilon)}\}}|\vartheta-\vartheta'|$. In the same manner build an $\epsilon$-net $x_1,\ldots,x_{N(\epsilon)}$. With this, since $\kappa$ is Lipschitz by assumption and $|x_{j(x)}-\vartheta_{j(\vartheta)}|\geq |x-\vartheta|-2\epsilon$, it holds for some constant $0<C<\infty$,
\begin{align*}
&\ \sup_{\vartheta\in[0,1],x\in[\vartheta+r_n,1]}\Big|\frac 1{nd_n}\sum_{i=1}^n\varepsilon_i\Big( \kappa\Big(\frac{X_i-x}{d_n}\Big)\Big(\frac{X_i-\vartheta}{x-\vartheta}\Big)^q-\kappa\Big(\frac{X_i-x_{j(x)}}{d_n}\Big)\Big(\frac{X_i-\vartheta_{j(\vartheta)}}{x_{j(x)}-\vartheta_{j(\vartheta)}}\Big)^q\Big)\Big|\\ 
\leq&\ \sup_{\vartheta\in[0,1],x\in[\vartheta+r_n,1]}\Big|\frac 1{nd_n}\sum_{i=1}^n\varepsilon_i\Big(\frac{X_i-\vartheta}{x-\vartheta}\Big)^q\Big( \kappa\Big(\frac{X_i-x}{d_n}\Big)-\kappa\Big(\frac{X_i-x_{j(x)}}{d_n}\Big)\Big)\Big|\\ 
&+\sup_{\vartheta\in[0,1],x\in[\vartheta+r_n,1]}\Big|\frac 1{nd_n}\sum_{i=1}^n\varepsilon_i\kappa\Big(\frac{X_i-x_{j(x)}}{d_n}\Big)\Big(\Big(\frac{X_i-\vartheta}{x-\vartheta}\Big)^q-\Big(\frac{X_i-\vartheta_{j(\vartheta)}}{x_{j(x)}-\vartheta_{j(\vartheta)}}\Big)^q\Big)\Big|\\ 
\leq&\  C\frac{\epsilon}{d_n^2}\frac 1n\sum_{i=1}^n |\varepsilon_i|\\
&+ \sup_{\vartheta\in[0,1],x\in[\vartheta+r_n,1]}\Big|\frac 1{nd_n}\sum_{i=1}^n\varepsilon_i \kappa\Big(\frac{X_i-x_{j(x)}}{d_n}\Big)\Big(\frac{X_i-\vartheta}{x-\vartheta}\Big)^q\,\frac{(x_{j(x)}-\vartheta_{j(\vartheta)})^q-(x-\vartheta)^q}{(x_{j(x)}-\vartheta_{j(\vartheta)})^q}\Big)\Big|\\
&+ \sup_{\vartheta\in[0,1],x\in[\vartheta+r_n,1]}\Big|\frac 1{nd_n}\sum_{i=1}^n\varepsilon_i \kappa\Big(\frac{X_i-x_{j(x)}}{d_n}\Big)\Big(\frac{x-\vartheta}{x_{j(x)}-\vartheta_{j(\vartheta)}}\Big)^q\,\frac{(X_i-\vartheta)^q-(X_i-\vartheta_{j(\vartheta)})^q}{(x-\vartheta)^q}\Big)\Big|\\
=&\ O\Big(\frac{\epsilon^{q\wedge 1}}{d_n^2(r_n-\epsilon)^{q\wedge 1}}\Big)\frac 1n\sum_{i=1}^n |\varepsilon_i| =\ O_{\P}(n^{-\frac 12}d_n^{-\frac 12})
\end{align*}
by the law of large numbers.
Further, we apply  a Bernstein inequality (e.\,g.\ Corollary A.8 in \cite{FerratyVieu}). Note that for all $l\geq 2$ it holds $E\big[\big(\frac 1{d_n}|\varepsilon|\kappa(\frac{x-X}{d_n})\big)^l\big]=O(d_n^{-(l-1)})$. Thus, for all $\eta>0$  and some constants $0<C,C'<\infty$ it holds
\begin{align*}
&\ \P\Big(\max_{\substack{j\in \{1,\ldots,N(\epsilon)\}\\k\in\{1,\ldots,N(\epsilon)\}:x_k>\vartheta_j+r_n-2\epsilon}}\Big|\frac 1{nd_n}\sum_{i=1}^n\varepsilon_i \Big( \kappa\Big(\frac{X_i-x_k}{d_n}\Big)\Big(\frac{X_i-\vartheta_j}{x_k-\vartheta_j}\Big)^q\Big|>\eta\sqrt{\frac{\log(N^2(\epsilon))}{nd_n}}\Big)\\
\leq&\ \ N^2(\epsilon)\max_{\substack{j\in \{1,\ldots,N(\epsilon)\}\\k\in\{1,\ldots,N(\epsilon)\}:x_k>\vartheta_j+r_n-2\epsilon}}\P\Big(\Big|\frac 1{nd_n}\sum_{i=1}^n\varepsilon_i \Big( \kappa\Big(\frac{X_i-x_k}{d_n}\Big)\Big(\frac{X_i-\vartheta_j}{x_k-\vartheta_j}\Big)^q\Big|>\eta\sqrt{\frac{\log(N^2(\epsilon))}{nd_n}}\Big)\\
\leq&\ N^2(\epsilon)2\exp\Big(-\frac{\eta^2\frac{\log(N^2(\epsilon))}{nd_n}n}{2\frac 1{d_n}(1+\eta \sqrt{\frac{\log(N^2(\epsilon))}{nd_n}})}\Big)\\
\leq&\ N^2(\epsilon)2\exp(-C'\eta^2\log(N^2(\epsilon)))\\
=&\ 2\exp((1-C'\eta^2)\log(N^2(\epsilon))
\end{align*}
which completes the proof of \eqref{regression2} since $\eta$ can be chosen such that $C'\eta^2>1$. 
\hfill$\Box$

\subsection{Proof of Lemma \ref{lem-deriv-pos} and Lemma \ref{lem-deriv}}

First, we consider 
\begin{align}
\hat m_\vartheta(x)-m(x)
=&\ 
\frac{\frac 1{nd_n}\sum_{i=1}^n Y_i\kappa\big(\frac{X_i-x}{d_n}\big)(X_i-\vartheta)_+^q(x-\vartheta)_+^q}{\frac 1{nd_n}\sum_{i=1}^n \kappa\big(\frac{X_i-x}{d_n}\big)(X_i-\vartheta)_+^{2q}}-m(x).\label{regression1}
\end{align}
If $x\leq \vartheta_0\wedge\vartheta$ the term \eqref{regression1} is zero. If $\vartheta_0\leq x\leq \vartheta$ it holds $\eqref{regression1}=-m(x)$. In Lemma \ref{lem-deriv-help} we have shown that,
for every $r_n\downarrow 0$ with $d_n=O(r_n)$,
$$\sup_{\vartheta\in[0,1],x\in[\vartheta+r_n,1]}\big|\hat m_\vartheta(x)-m(x)\big|=O_{\mathbb{P}}(d_n^{q\wedge 1}r_n^{-(q\wedge 1)})+O_{\mathbb{P}}\Big(\sqrt{\frac{\log(n)}{nd_n}}\Big).$$
Thus, applying these bounds for \eqref{regression1} on $|M_n-M|$, we get in the case of Lemma \ref{lem-deriv-pos}
\begin{align}
&\sup_{\vartheta\in[0,1]}|M_n(\vartheta)-M(\vartheta)|\nonumber\\
=&\ \sup_{\vartheta\in[0,1]}\bigg|\frac 1n\sum_{j=1}^n \left(Y_j-\hat m_{\vartheta}(X_j)\right)-\int_{\vartheta_0}^{\vartheta}m(x)dF_X(x)I\{\vartheta_0<\vartheta\}\bigg|\nonumber\\
=&\ \sup_{\vartheta\in[0,1]}\bigg|\frac 1n\sum_{j=1}^n \Big(m(X_j)+\varepsilon_j- \hat m_\vartheta(X_j)\Big)
-\int_{\vartheta_0}^{\vartheta}m(x)dF_X(x)I\{\vartheta_0<\vartheta\}\bigg|\nonumber\\
\leq&\  \underbrace{\bigg|\frac 1n\sum_{j=1}^n\varepsilon_j\bigg|}_{=O_{\P}(n^{-\frac 12})}\nonumber\\
&+\sup_{\vartheta\in[\vartheta_0,1]}\bigg|\frac 1n\sum_{j=1}^n m(X_j)I\{\vartheta_0\leq X_j\leq \vartheta\}-\int_{\vartheta_0}^{\vartheta}m(x)dF_X(x)\bigg|\nonumber\\
&+\sup_{\vartheta\in[0,1]} \bigg|\frac 1n\sum_{j=1}^n \bigg(m(X_j)-\frac{\frac 1{nd_n}\sum_{i=1}^n Y_i\kappa\big(\frac{X_i-X_j}{d_n}\big)(X_i-\vartheta)_+^q(X_j-\vartheta)_+^q}{\frac 1{nd_n}\sum_{i=1}^n \kappa\big(\frac{X_i-X_j}{d_n})(X_i-\vartheta)_+^{2q}}\bigg)\nonumber\\
&\hspace{10cm}I\{\vartheta\leq X_j\leq \vartheta+r_n\}\bigg|\label{regression5}\\
&+O_{\mathbb{P}}(d_n^{q\wedge 1}r_n^{-(q\wedge 1)})+O_{\P}\Big(\sqrt{\frac{\log(n)}{nd_n}}\Big)\nonumber,
\end{align}
and very similar in the case of Lemma \ref{lem-deriv}
\begin{align}
&\sup_{\vartheta\in[0,1]}|M_n(\vartheta)-M(\vartheta)|\nonumber\\
=&\ \sup_{\vartheta\in[0,1]}\bigg|\frac 1n\sum_{j=1}^n \left(Y_j-\hat m_{\vartheta}(X_j)\right)^2-\hat\sigma_n^2-\int_{\vartheta_0}^{\vartheta}m^2(x)dF_X(x)I\{\vartheta_0<\vartheta\}\bigg|\nonumber\\
=&\ \sup_{\vartheta\in[0,1]}\bigg|\frac 1n\sum_{j=1}^n \Big(m(X_j)+\varepsilon_j- \hat m_\vartheta(X_j)\Big)^2
-\hat\sigma_n^2-\int_{\vartheta_0}^{\vartheta}m^2(x)dF_X(x)I\{\vartheta_0<\vartheta\}\bigg|\nonumber\\
\leq&\ 2 \underbrace{\bigg|\frac 1n\sum_{j=1}^n\varepsilon_j^2-\hat\sigma_n^2\bigg|}_{=O(|\hat\sigma_n^2-\sigma^2|)+O_{\P}(n^{-\frac 12})}\nonumber\\
&+2\sup_{\vartheta\in[\vartheta_0,1]}\bigg|\frac 1n\sum_{j=1}^n m^2(X_j)I\{\vartheta_0\leq X_j\leq \vartheta\}-\int_{\vartheta_0}^{\vartheta}m^2(x)dF_X(x)\bigg|\nonumber\\
&+2\sup_{\vartheta\in[0,1]} \bigg|\frac 1n\sum_{j=1}^n \bigg(m(X_j)-\frac{\frac 1{nd_n}\sum_{i=1}^n Y_i\kappa\big(\frac{X_i-X_j}{d_n}\big)(X_i-\vartheta)_+^q(X_j-\vartheta)_+^q}{\frac 1{nd_n}\sum_{i=1}^n \kappa\big(\frac{X_i-X_j}{d_n})(X_i-\vartheta)_+^{2q}}\bigg)^2\nonumber\\
&\hspace{10cm}I\{\vartheta\leq X_j\leq \vartheta+r_n\}\bigg|\label{regression3}\\
&+O_{\mathbb{P}}(d_n^{2(q\wedge 1)}r_n^{-2(q\wedge 1)})+O_{\P}\Big(\frac{\log(n)}{nd_n}\Big)\nonumber.
\end{align}
The function classes
$$\big\{m(\cdot)I\{\vartheta_0\leq \cdot\leq \vartheta\}:\vartheta\in[\vartheta_0,1]\big\}\text{ and }\big\{m^2(\cdot)I\{\vartheta_0\leq \cdot\leq \vartheta\}:\vartheta\in[\vartheta_0,1]\big\}$$
are Donsker. The process is very similar to the classical empirical distribution function and we can e.\,g.\ adapt example 19.6 in \cite{vanderVaart} to prove the Donsker property.
Thus, with the continuous mapping theorem, we get $\sup_{\vartheta\in[\vartheta_0,1]}\big|\frac 1n\sum_{j=1}^n m(X_j)I\{\vartheta_0\leq X_j\leq \vartheta\}-\int_{\vartheta_0}^{\vartheta}m(x)dF_X(x)\big|=O_{\P}(n^{-1/2})$ and the same with $m$ replaced by $m^2$. Analogously we get $\sup_{\vartheta\in[\vartheta_0,1]}\big|\frac 1n\sum_{j=1}^n I\{\vartheta\leq X_j\leq \vartheta+r_n\}-\int_{\vartheta}^{\vartheta+r_n}dF_X(x)\big|=O_{\P}(n^{-1/2})$. Thus, the terms \eqref{regression5} and \eqref{regression3} can be bounded by
\begin{align*}
2\big(2\sup_{x\in[0,1]}|m(x)|+\max_{i\in\{1,\ldots,n\}}|\varepsilon_i|\big)^2\cdot\Big(r_n+O_{\P}(n^{-1/2})\Big)=O_{\P}(\log(n)r_n),
\end{align*}
where the last equality follows with $\max_{i\in\{1,\ldots,n\}}|\varepsilon_i|^2=O_{\P}(\log(n))$, which can be shown again with a Bernstein inequality.
Choosing $r_n=d_n^{(q\wedge 1)/(1+(q\wedge 1))}\log(n)^{-1/(1+(q\wedge 1))}$ for Lemma \ref{lem-deriv-pos} and $r_n=d_n^{2(q\wedge 1)/(1+2(q\wedge 1))}\log(n)^{-1/(1+2(q\wedge 1))}$ for Lemma \ref{lem-deriv} we get the asserted rates.
\hfill$\Box$

\subsection{Proof of Theorem \ref{theoregr}}

From Lemma \ref{lem-deriv-help} we get
\begin{eqnarray*}
\sup_{x\in[\hat\vartheta_n+r_n,1]}\big|\hat m_n(x)-m(x)\big|
&=&\ \sup_{x\in[\hat\vartheta_n+r_n,1]}\big|\hat m_{\hat\vartheta_n}(x)-m(x)\big|\\
&\leq&\ \sup_{\vartheta\in[0,1]}\sup_{x\in[\vartheta+r_n,1]}\big|\hat m_{\vartheta}(x)-m(x)\big|\\
&=&\ O_{\mathbb{P}}(d_n^{q\wedge 1}r_n^{-(q\wedge 1)})+O_\mathbb{P}\big(n^{-1/2}d_n^{-1/2}\log(n)^{1/2}\big).
\end{eqnarray*}
For $x\leq \hat\vartheta_n\wedge \vartheta_0$ it holds $\hat m_n(x)=m(x)=0$, and for $\vartheta_0\leq x\leq \hat\vartheta_n$ $$\sup_{x\in[\vartheta_0,\hat\vartheta_n]}|\hat m_{\hat\vartheta_n}(x)-m(x)|= \sup_{x\in[\vartheta_0,\hat\vartheta_n]}|m(x)|\leq\sup_{x\in[0,1]}|g(x)|(\hat\vartheta_n-\vartheta_0)^q=O_\mathbb{P}(\delta_n^q).$$
\hfill $\Box$

\begin{rem}
For $x\in(\hat\vartheta_n,\hat\vartheta_n+r_n)$ with $r_n$ from Theorem \ref{theoregr} it holds $$|m(x)|\leq \sup_{x\in[0,1]}|g(x)|(\hat\vartheta_n+r_n-\vartheta_0)_+^q=O_\mathbb{P}((r_n+\delta_n)^q)=O_\mathbb{P}(r_n^q)+O_\mathbb{P}(\delta_n^q),$$
and for $X_i\leq x+C_\kappa d_n$ it follows $|m(X_i)|=O_\mathbb{P}((r_n+\delta_n+d_n)^q)=O_\mathbb{P}(r_n^q)+O_\mathbb{P}(\delta_n^q)$. \\
 Moreover $\sup_{x\in[0,1]}\bigg|\frac{\frac 1{nd_n}\sum_{i=1}^n \kappa\big(\frac{X_i-x}{d_n}\big)(X_i-\hat\vartheta_n)_+^q(x-\hat\vartheta_n)_+^q}{\frac 1{nd_n}\sum_{i=1}^n \kappa\big(\frac{X_i-x}{d_n}\big)(X_i-\hat\vartheta_n)_+^{2q}}\bigg|\leq 1$.
Thus,
\begin{align*}
&\sup_{x\in[0,1]}|\hat m_n(x)-m(x)|\\\leq&\ O_\mathbb{P}(r_n^q)+O_{\P}(n^{-1/2}d_n^{-1/2}\log(n)^{1/2})+O_{\P}(\delta_n^q)+O_{\P}(d_n^{q\wedge 1}r_n^{-(q\wedge 1)})\\
&+\underbrace{\sup_{x\in(\hat\vartheta_n,\hat\vartheta_n+r_n)}\bigg|\frac{\frac 1{nd_n}\sum_{i=1}^n \varepsilon_i\kappa\big(\frac{X_i-x}{d_n}\big)(X_i-\hat\vartheta_n)_+^q(x-\hat\vartheta_n)_+^q}{\frac 1{nd_n}\sum_{i=1}^n \kappa\big(\frac{X_i-x}{d_n}\big)(X_i-\hat\vartheta_n)_+^{2q}}\bigg|}_{=O\Big( \sup_{x\in(\hat\vartheta_n,\hat\vartheta_n+r_n)}\Big|\frac 1{nd_n}\sum_{i=1}^n \varepsilon_i\kappa\big(\frac{X_i-x}{d_n}\big)\Big(\frac{X_i-\hat\vartheta_n}{x-\hat\vartheta_n}\Big)^q\Big|\Big)}.
\end{align*}
Choosing $r_n=d_n^{(q\wedge 1)/(q+(q\wedge 1))}$ we get 
\begin{align*}
&\sup_{x\in[0,1]}|\hat m_n(x)-m(x)|\\\leq&\ O_\mathbb{P}(d_n^{\frac{q(q\wedge 1)}{q+(q\wedge 1)}})+O_\mathbb{P}(\delta_n^q)+O_{\P}(n^{-1/2}d_n^{-1/2}\log(n)^{1/2})\\
&+\sup_{x\in(\hat\vartheta_n,\hat\vartheta_n+d_n^{(q\wedge 1)/(q+(q\wedge 1))})}\bigg|\frac{\frac 1{nd_n}\sum_{i=1}^n \varepsilon_i\kappa\big(\frac{X_i-x}{d_n}\big)(X_i-\hat\vartheta_n)_+^q(x-\hat\vartheta_n)_+^q}{\frac 1{nd_n}\sum_{i=1}^n \kappa\big(\frac{X_i-x}{d_n}\big)(X_i-\hat\vartheta_n)_+^{2q}}\bigg|.
\end{align*}
\end{rem}

\subsection{Proof of Corollary \ref{cor:regression_fixed}}

For $x=\vartheta_0$ we get
$$\hat m_n(x)-m(x)=\hat m_n(\vartheta_0)=\frac{\frac{1}{nd_n^{2q+1}}\sum_{i=1}^n Y_i\kappa\big(\frac{X_i-\vartheta_0}{d_n}\big)(X_i-\hat\vartheta_n)_+^q}{\frac{1}{nd_n^{2q+1}}\sum_{i=1}^n \kappa\big(\frac{X_i-\vartheta_0}{d_n}\big)(X_i-\hat\vartheta_n)_+^{2q}}(\vartheta_0-\hat\vartheta_n)_+^q.$$
Due to the factor one only considers the case $\vartheta_0>\hat\vartheta_n$, and due to the one-sided kernel only the covariates $X_i>\vartheta_0$. Thus $X_i-\hat\vartheta_n\geq X_i-\vartheta_0$, and the expectation of the non-negative denominator has  a lower bound
$$
\frac{1}{d_n^{2q+1}}\int \kappa\big(\frac{x-\vartheta_0}{d_n}\big)(x-\vartheta_0)_+^{2q}f_X(x)\, dx =\int \kappa(u) u^{2q}f_X(\vartheta_0+d_nu)\,du =c+o(1)$$
under the assumption $c= \int \kappa(u) u^{2q}\,du>0$ and $f_X$ is bounded and continuous. The numerator can be written as $A_n+B_n+R_n$
with 
\begin{eqnarray*}
    A_n &=& \frac{1}{nd_n^{2q+1}}\sum_{i=1}^n m(X_i)\kappa\big(\frac{X_i-\vartheta_0}{d_n}\big)(X_i-\vartheta_0)_+^q\\
    B_n &=& \frac{1}{nd_n^{2q+1}}\sum_{i=1}^n \varepsilon_i\kappa\big(\frac{X_i-\vartheta_0}{d_n}\big)(X_i-\vartheta_0)_+^q\\
    R_n &=& \frac{1}{nd_n^{2q+1}}\sum_{i=1}^n Y_i\kappa\big(\frac{X_i-\vartheta_0}{d_n}\big)\left((X_i-\hat\vartheta_n)_+^q-(X_i-\vartheta_0)_+^q\right)
\end{eqnarray*}
and
$$|R_n|\leq \frac{1}{nd_n^{2q+1}}\sum_{i=1}^n |Y_i|\kappa\big(\frac{X_i-\vartheta_0}{d_n}\big)|\hat\vartheta_n-\vartheta_0|^q=O_\P(\frac{\delta_n^q}{d_n^{2q}})$$
in the case $q\in (0,1]$, and 
$$|R_n|\leq \frac{1}{nd_n^{2q+1}}\sum_{i=1}^n |Y_i|\kappa\big(\frac{X_i-\vartheta_0}{d_n}\big)|\hat\vartheta_n-\vartheta_0|=O_\P(\frac{\delta_n}{d_n^{2q}})$$
in the case $q>1$. Further 
$$E|A_n|\leq \frac{1}{d_n^{2q+1}}\int \kappa\big(\frac{x-\vartheta_0}{d_n}\big)(x-\vartheta_0)_+^{2q}|g(x)|f_X(x)\, dx=O(1),$$
and $E[B_n]=0$, 
$$\Var(B_n)=\frac{1}{nd_n^{4q+2}}\int \kappa^2\big(\frac{x-\vartheta_0}{d_n}\big)(x-\vartheta_0)_+^{2q}\sigma^2(x)f_X(x)\, dx=O(\frac{1}{nd_n^{2q+1}}).$$
Altogether we obtain 
$$\hat m_n(x)-m(x)=\hat m_n(\vartheta_0)=\left(O_\P(1)+O_\P\Big(\frac{1}{n^{1/2}d_n^{q+1/2}}\Big)+O_\mathbb{P}\Big(\frac{\delta_n^{q\wedge 1}}{d_n^{2q}}\Big)\right)O_\P(\delta_n^q).$$
\hfill $\Box$

\end{appendix}


\bibliographystyle{plainnat}
\bibliography{References.bib}

\end{document}